\documentclass[12pt]{amsart}
\usepackage{amsmath}
\usepackage{amssymb}
\usepackage{wasysym}
\usepackage{graphicx}
\usepackage{xcolor}
\usepackage{hyperref}
\numberwithin{equation}{section}
\usepackage[all]{xy}

\newcommand{\lra}{\longrightarrow}
\newcommand{\ra}{\rightarrow}

\newcommand{\blambda}{\boldsymbol{\lambda}}
\newcommand{\bmu}{\boldsymbol{\mu}}

\newcommand{\bA}{\mathbb{A}}
\newcommand{\bF}{\mathbb{F}}
\newcommand{\bG}{\mathbb{G}}
\newcommand{\bP}{\mathbb{P}}

\newcommand{\bZ}{\mathbb{Z}}

\newcommand{\cA}{\mathcal{A}}
\newcommand{\cE}{\mathcal{E}}
\newcommand{\cL}{\mathcal{L}}
\newcommand{\cK}{\mathcal{K}}
\newcommand{\cO}{\mathcal{O}}
\newcommand{\cS}{\mathcal{S}}

\newcommand{\fS}{\mathfrak{S}}

\newcommand{\rH}{\mathrm{H}}

\newcommand{\oX}{\overline{X}}

\newcommand{\wG}{\widehat{G}}

\newcommand{\wQ}{\widehat{q}}
\newcommand{\wT}{\widehat{\tau}}

\newcommand{\Br}{\mathrm{Br}}
\newcommand{\Bl}{\mathrm{Bl}}

\newcommand{\cok}{\operatorname{cok}}
\newcommand{\Cone}{\operatorname{Cone}}
\newcommand{\disc}{\operatorname{disc}}
\newcommand{\End}{\operatorname{End}}
\newcommand{\GL}{\mathrm{GL}}
\newcommand{\Gr}{\mathrm{Gr}}
\newcommand{\GO}{\mathrm{GO}}
\newcommand{\OO}{\mathrm{O}}
\newcommand{\Hom}{\mathrm{Hom}}
\newcommand{\OGr}{\mathrm{OGr}}

\newcommand{\oPSO}{\overline{\mathrm{PSO}}}
\newcommand{\PGL}{\mathrm{PGL}}
\newcommand{\Pic}{\mathrm{Pic}}
\newcommand{\PGO}{\mathrm{PGO}}
\newcommand{\PSO}{\mathrm{PSO}}
\newcommand{\PWitt}{\mathrm{PWitt}}
\newcommand{\rank}{\operatorname{rank}}
\newcommand{\SO}{\mathrm{SO}}
\newcommand{\Spin}{\operatorname{Spin}}

\newcommand{\Sym}{\mathrm{Sym}}

\newcommand{\spa}{\mathrm{span}}
\newcommand{\SWitt}{\mathrm{SWitt}}
\newcommand{\Witt}{\mathrm{Witt}}

\theoremstyle{plain}
\newtheorem{prop}{Proposition}[section]

\newtheorem{coro}[prop]{Corollary}

\newtheorem{assu}[prop]{Assumption}

\theoremstyle{definition}
\newtheorem{defi}[prop]{Definition}

\newtheorem{ques}[prop]{Question}

\theoremstyle{remark}
\newtheorem{rema}[prop]{Remark}

\newtheorem{exam}[prop]{Example}

\title[Witt rings and Pfister forms]{Witt rings, Pfister forms, and equivariant birational geometry}

\author{Brendan Hassett}
\address{Department of Mathematics\\
Brown University\\
Box 1917 \\
151 Thayer Street,
Providence, RI 02912 \\
USA}
\email{brendan\underline{ }hassett@brown.edu}

\author{Yuri Tschinkel}
\address{
  Courant Institute \\
  251 Mercer Street\\
  New York, NY 10012, USA \\
\indent  Simons Foundation\\
160 Fifth Avenue\\
New York, NY 10010,
USA }
\email{tschinkel@cims.nyu.edu}

\begin{document}

\begin{abstract}
    We study equivariant birational geometry of quadrics with actions of finite groups and develop analogs of Witt groups and Pfister theory in the equivariant context. 
\end{abstract}

\maketitle

\date{\today}

\section{Introduction}

The birational geometry of smooth quadric hypersurfaces over 
a non-closed field is a beautiful and intricate subject. 
Generic splitting fields, essential dimension, motivic cohomology,
and the Milnor conjectures interact to form a deep and varied
theory. We refer the reader to \cite{TotBGQ} for an overview
of birational questions, \cite{DR,KMEDQ} for important insights
coming from versality constructions, and \cite{EKMbook,kahn-book}
for a comprehensive overview. At the same time, there are intriguing results for quadrics in specific dimensions, see, e.g.,  \cite{Hoffmann5D,Karpenko-survey}.

Our focus here is on quadrics with generically free actions of
finite groups $G$ and their equivariant birational geometry. 
Quadrics over fields and over classifying spaces $BG$ have many
common features, especially, where cohomological invariants
are concerned.  However, birational questions over $BG$ are not
as well-understood, with most results focusing on examples 
of small dimensions, see \cite{HTEGLDQ}, \cite[Section 4]{CTZ}. 

To move beyond low dimensions, we aim to translate
ideas and techniques from the theory of quadratic forms 
(of arbitrary dimension!) over fields to the equivariant context.
Here we consider Witt rings and Pfister forms. One complication is
the distinction between stable and unstable notions;
for us, stabilization is taking the product with a linear representation
of $G$. At an elementary level, rational varieties over infinite
fields have many rational points, which we can exploit for geometric
constructions. Witt's theory of isotropic subspaces is the
archetypal example: Once we have an isotropic subspace, we can produce
many more of the same dimension via explicit operations.
In the equivariant context, it is necessary to consider {\em all} such
constructions simultaneously; even linear actions of non-abelian
groups on projective space may lack fixed points.  However, this
presents no difficulty when the parameter space itself
is stably linearizable.

Section~\ref{sect:equiv} establishes a bridge between equivariant
stable birational geometry and birational geometry over non-closed fields.
We review quadrics over fields in Section~\ref{sect:quad} and
wonderful compactifications in Section~\ref{sect:compactify}; 
this machinery sheds new light on multiplicative quadratic
forms. Equivariant
formulations are developed in Section~\ref{sect:equiv-quad}.
Section~\ref{sect:inv} explores birational implications of 
invariant isotropic subspaces; stable birational properties are revealed
on restriction to $2$-Sylow subgroups. Several notions of
Witt rings are presented in Section~\ref{sect:wring}.
Stably Pfister forms are treated in Section~\ref{sect:equipf};
their unstable analogues are explored 
in Section~\ref{sect:unstable}.

\medskip
\noindent
\textbf{Acknowledgments:}
The first author was partially supported by the Simons Foundation
Award 546235 and the US National Science Foundation Grant 1929284.
The second author was partially supported by NSF grant 2301983.
We are grateful to Candace Bethea for conversations that
informed this project, and to Barry Mazur for suggestions on a draft
of this paper. We used MAGMA \cite{MAGMA} and Gemini for computations informing our results, and ChatGPT
and Claude to proofread the text. 

\section{Equivariant birational geometry}
\label{sect:equiv}
We recall main terms and constructions in equivariant birational geometry; for a more detailed discussion, see, e.g., \cite[Sec. 2]{HTPAMQ}.

Let $k$ be a field of characteristic zero and $X$ a smooth projective variety over $k$ with a regular action by a finite group $G$.
The action is {\em $G$-equivariantly stably linearizable} (or {\em stably rational}) if there exist $G$-representations $V$
and $W$ and a $G$-equivariant birational map
$$X \times V \stackrel{\sim}{\dashrightarrow} W.$$ 
When the action on $X$ is generically free, the No-Name Lemma implies that this is equivalent
to the existence of an equivariant birational map
$$X  \times \bA^d \stackrel{\sim}{\dashrightarrow} W,$$
where the action on $\bA^d$ is trivial. It is also equivalent to the existence 
of a $G$-equivariant vector bundle $E \rightarrow X$ and an equivariant
$$E \stackrel{\sim}{\dashrightarrow} W.$$

Two smooth projective varieties $X_1$ and $X_2$ with $G$-actions are {\em $G$-equivariantly stably birational} if there exist representations $W_1$ and $W_2$ and a $G$-equivariant birational
$$X_1 \times W_1 \stackrel{\sim}{\dashrightarrow} X_2 \times W_2.$$
This is an equivalence relation.  
Again, if the actions on $X_1$ and $X_2$ are generically free, this follows from 
the existence of a $G$-equivariant birational 
$$E_1 \stackrel{\sim}{\dashrightarrow} E_2,$$
where $E_1$ and $E_2$ are $G$-equivariant vector bundles over $X_1$ and $X_2$, 
respectively.

The work of Duncan and Reichstein \cite{DR} provides links between equivariant
birational geometry and geometry over non-closed fields. In particular, we can test equivariant
stable birationality using twisted forms over function fields. 
We record the following strengthening of \cite[Th. 1.1]{DR}:

\begin{prop} 
\label{prop:dr}
Fix a linear algebraic group $G$ and geometrically irreducible varieties 
$X_1$ and $X_2$, with 
regular $G$-actions. Then the following are equivalent:
\begin{itemize} 
\item[(1)] $X_1$ and $X_2$ are $G$-equivariantly stably birational;
\item[(2)] for every $G$-torsor $T$
defined over any extension $K/k$, the twists
$${ }^T X_1, { }^TX_2$$
are stably birational over $K$;
\item[(3)] there exists a linear representation
$V$ of $G$ for which $G$ acts freely on a dense open subset
$V_0 \subset V$, with field of invariants $F=k(V)^G$,
such that the twists
$$
{ }^{V_0} X_1, { }^{V_0}X_2
$$
are stably birational over $F$.
\end{itemize}
\end{prop}
Similar reasoning may be found in \cite[{\S}2]{KrTs26}.

\begin{proof}  
The implication $(2) \Rightarrow (3)$ is immediate,
as $V_0$ is a $G$-torsor over $V_0/G$. 
Below we use $S$ to denote the corresponding torsor
over $F$.

We prove $(1) \Rightarrow (2)$.  
Since $X_1$ and $X_2$ are equivariantly stably
birational, there exist $G$-representations $W_1$ and $W_2$
and an equivariant birational
$$\phi: X_1 \times W_1 \stackrel{\sim}{\dashrightarrow}
X_2 \times W_2.$$
Twisting via $T$, we obtain
$${ }^T\phi: { }^TX_1 \times_K { }^TW_1 
\stackrel{\sim}{\dashrightarrow} { }^TX_2 \times_K { }^TW_2.$$
By Hilbert's Theorem 90, the twisted representations
are $K$-rational; it follows that ${ }^TX_1$ and ${ }^TX_2$
are stably birational over $K$.

We turn to $(3) \Rightarrow (1)$.  We have birational maps
$$
{ }^{V_0} X_1 \times \bA^e_F \stackrel{\sim}{\dashrightarrow}
{ }^{V_0}X_2 \times \bA^d_F.
$$
We also have
$${ }^{V_0} X_1 \times \bA^e_F \simeq { }^{V_0}(X_1 \times \bA^e_F)
$$
with trivial action on the last $\bA^e_F$.  

We recall \cite[Lem.~5.1]{DR}: Suppose that $X$ has a $G$-action. A birational map
$$
Y_F \stackrel{\sim}{\dashrightarrow} { }^{V_0}X 
$$
yields a $G$-equivariant birational 
$$V \times Y \stackrel{\sim}{\dashrightarrow} V \times X.
$$

Applying this, we get $G$-equivariant birational
$$V \times \bA^e_k \times X_1 
\stackrel{\sim}{\dashrightarrow}
V \times \bA^d_k \times X_2$$
with trivial actions on the affine spaces.  
It follows that $X_1$ and $X_2$ are $G$-equivariantly stably birational.
\end{proof}

\section{Quadrics over fields}
\label{sect:quad}

Assume that the base field $k$ has characteristic not equal to two.

\subsection*{Background on quadratic and bilinear forms}
Let $V$ be a finite-dimensional vector space over $k$. 
A symmetric bilinear form
$$b:V \times V \rightarrow k$$
is {\em non-degenerate} if the associated linear map
$$\begin{array}{rcl}
    s:V& \ra & V^{\vee} \\
    v & \mapsto & b(v,-)
    \end{array}
$$ 
is an isomorphism; its associated quadratic form is 
$$
q(v):=b(v,v).
$$
Its {\em rank} is the dimension of $V$.
Given symmetric bilinear forms $b_i:V_i\times V_i \rightarrow k$
we may define their direct sum
\begin{align*}
&\,\,b_1\oplus b_2: (V_1 \oplus V_2) \times (V_1 \oplus V_2) \rightarrow k,\\
&(b_1\oplus b_2)(v_1+v_2, v'_1+v'_2)=b_1(v_1,v_1') + b_2(v_2,v_2')
\end{align*}
and tensor product
\begin{align*}
 & \,\,   b_1\otimes b_2: (V_1 \otimes V_2) \times (V_1 \otimes V_2) \rightarrow k,
    \\
& (b_1\otimes b_2)(v_1\otimes v_2, v'_1\otimes v'_2)=b_1(v_1,v_1') b_2(v_2,v_2').
\end{align*}

A linear map $g:V \rightarrow V$ is a {\em similitude} with respect to $b$ if
there is a scalar $\lambda(g)$ with
$$b(g(v),g(v'))= \lambda(g) b(v,v'), \quad \text{ for all }v,v' \in V.$$
These form a group $\GO(V,b)$ with a natural character
$$\lambda: \GO(V,b) \rightarrow \bG_m;$$
its kernel is denoted $\OO(V,b)$ and elements of determinant one
is written $\SO(V,b)$. There is analogous
notation for quadratic forms $q$.
We may interpret $q \in \Sym^2(V^{\vee})\otimes L$,
where $L$
is the one-dimensional representation associated with $\lambda$.
In the non-degenerate case, we have \cite[\S 12.A]{KMRT}
$$\det(g)^2 = \lambda(g)^{\dim(V)}.$$

\subsection*{Discriminant}
Assuming further that $\dim(V)=2m$, we obtain a character
$$\disc:\GO(V,b) \ra \{ \pm 1\}, \quad \disc(g)=\det(g) \lambda(g)^{-\dim(V)/2}.$$
This coincides with the determinant on 
restriction to $\OO(V,b)$.

\

Given a non-degenerate quadratic form $(V,q)$, an {\em isotropic subspace} 
$W\subset V$ is one where $q|W=0$;  we have $\ell:=\dim(W) \le \dim(V)/2$. 
The {\em isotropic
Grassmannian} 
$$
\OGr(\ell,q)
$$
is the variety of all 
$\ell$-dimensional isotropic subspaces. This scheme is non-empty 
for $\ell \le \dim(V)/2$.  These come with a $\GO(V,q)$ action.

\subsection*{Clifford algebras and spin: Even-dimensional case}

We recall basic facts from \cite[\S 85]{EKMbook}:
Let $(V,q)$ be non-degenerate with $\dim(V)=2m$ and $\OGr(m,q)$ its
maximal isotropic Grassmannian. Geometrically, it has two connected components $\OGr(m,q)_{\pm}$, each smooth
of dimension $(m-1)m/2$, permuted according to discriminant character.
Over the discriminant extension, each connected component of $\OGr(m,q)$ is a point for $m=1$, a conic when $m=2$,
and a Brauer-Severi threefold for $m=3$.

By \cite[pp. 387-390]{FulHar}, over an algebraically closed field, we have embeddings 
$$
\OGr(m,q)_{\pm} \hookrightarrow \bP(S_{\pm}),
$$
where the $S_{\pm}$ are the {\em half-spin representations}, which are irreducible $2^{m-1}$-dimensional representations of the spin group
\begin{equation}
\label{eqn:spin}
1 \rightarrow \mu_2 \ra \Spin(q) \rightarrow \SO(V,q) \rightarrow 1.
\end{equation}
The even Clifford algebra decomposes \cite[p.~305]{FulHar}
$$
C_0(q) \simeq \End(S_+) \times \End(S_-).
$$

Over a nonclosed field, we have:
\begin{prop} \cite[Rem.~13.9]{EKMbook} 
\label{prop:spinAmitsur}
Let $(V,q)$ be a non-degenerate quadratic form over a field $k$, with $\dim(V)=2m$ and trivial discriminant.
Then 
$$C_0(q) \simeq C^+(q) \times C^-(q)$$
and the Brauer-Severi varieties associated with $\bP(S_{\pm})$ have the same class in $\Br(k)$, realized via Clifford invariant 
$c(q) \in \rH^2(k,\mu_2)$ arising from the connecting homomorphisms of (\ref{eqn:spin}).
\end{prop}

We elaborate on Clifford structures and the actions of similitudes \cite[\S 13.A]{KMRT}. 
Recall that $L$ is the representation of $\GO(V,q)$ associated with $\lambda$.
Given $q \in \Sym^2(V^{\vee})\otimes L$ we may construct the {\em even Clifford algebra}
$$C_0(q) = \left( \oplus_{i=0}^m (V^{\otimes 2i}\otimes L^{-i}) \right)/I, \quad 2m=\dim(V),$$
where the ideal $I$ comes from the multiplication induced by
$$v^2 \ell \mapsto \ell(q(v)), \quad v\in V, \ell \in L^{\vee},
$$
i.e.,
$$I=\left< v^2 \otimes \ell - \ell(q(v^2)) \right>.$$
Thus we have an isomorphism of $\GO(V,q)$-representations
$$C_0(q) = \oplus_{i=0}^m (\wedge^{2i} V \otimes L^{-i})$$
as well as an action of $\GO(V,q)$. 
When $\dim(V)=2m$, its center
is a quadratic algebra containing $1$, determining
the discriminant invariant \cite[Prop. 11.6]{EKMbook}.
.

\subsection*{Clifford algebras and spin: Odd-dimensional case}
We now consider $(V,q)$ with $\dim(V)=2m+1$. 
Over an algebraically closed field, $C_0(q) \simeq \End(S)$
where $S$ is the irreducible spin representation of the associated
spin group \cite[Prop.~20.20]{FulHar}.
We have an embedding \cite[p.~390]{FulHar}
$$\OGr(m,q) \hookrightarrow \bP(S)$$
such that the square of the ample line bundle equals
the polarization from the Pl\"ucker embedding.  

Over a general field, $C_0(q)$ is a central simple algebra
\cite[Prop.~11.6(2)]{EKMbook}, also called 
the Clifford invariant of $(V,q)$,
and $\OGr(m,q)$ embeds into the corresponding Brauer-Severi
variety of dimension $2^m-1$.  
When $m=2$, the embedding is an isomorphism and 
$\OGr(2,q)$ is a Brauer-Severi threefold.

\subsection*{Splitting varieties of linear subspaces}
\begin{prop}
\label{prop:linalg}
Let $(V,q)$ be a non-degenerate quadratic form over a field $k$.
\begin{itemize}
\item{If $\dim(V)=2m-1$ and $\OGr(m-1,q)$ has a $k$-rational point then it is rational.}
\item{If $\dim(V)=2m$ and $\OGr(m,q)$ has a $k$-rational
point then the irreducible component of $\OGr(m,q)$ 
containing that point is rational over $k$.}
\end{itemize}
\end{prop}
\begin{proof}
We establish the first statement:
Let $L \subset V$ denote an isotropic subspace of dimension
$m-1$ over $k$. Linear algebra gives another isotropic
subspace $M \subset V$ of dimension $m-1$ such that
the induced pairing between $L$ and $M$ is non-degenerate
i.e., $M\simeq L^{\vee}$. Write
$$L^{\perp} \cap M^{\perp} = \spa(w)$$
for some non-zero $w \in V$.
Consider elements
$$A \in \Hom(L,M\oplus k w)$$
whose projections to $\Hom(L,M)\simeq \Hom(L,L^{\vee})$
are anti-symmetric under duality. 
This anti-symmetry is equivalent to the graph
$$\Gamma_A \subset L \oplus M \oplus k w = V$$
being isotropic for $q$.
The generic element of $\OGr(m-1,q)$ arises from such
a graph, giving a birational
$$\bA^{\binom{m-1}{2} + m-1} \stackrel{\sim}{\dashrightarrow}
\Gr(m-1,q),$$
and the desired rationality.

Now for the second assertion:
Let $L \in \OGr(m,q)$ be the rational point;
label the irreducible components so that 
$L \in \OGr(m,q)_+$.
Using linear algebra, we find an isotropic subspace
$M \subset V$ of dimension $m$ such that the
induced pairing between $L$ and $M$ is non-degenerate, i.e. $M\simeq L^{\vee}$. Note that $M \in \OGr(m,q)_+$
if $m$ is even and $M \in \OGr(m,q)_-$ 
when $m$ is odd.  Consider elements
$$A \in \Hom(L,M)\simeq \Hom(L,L^{\vee})$$
that are anti-symmetric under duality, which is equivalent to the graph
$$\Gamma_A \subset L \oplus M = V$$
being isotropic for $q$. Thus 
$\OGr(m,q)_+$ is birational to an affine
space of dimension $m(m-1)/2$. 
\end{proof}
For odd $m$, the same argument with $L$
and $M$ interchanged implies that $\OGr(m,q)_-$
is also rational. 


\begin{prop}
\label{prop:vq}
Let $(V,q)$ be a non-degenerate quadratic form
over a field $k$. 
Suppose $\dim(V)=2m$ and $\OGr(m,q)_{\pm}$ are both
defined over $k$, i.e., the discriminant of $q$
is trivial. Then $\OGr(m,q)_+$ and $\OGr(m,q)_-$
are isomorphic to each other. 
\end{prop}
In particular, if one component has a rational point then the
the other does as well.
\begin{proof}
Diagonalize 
$$q=\sum_{i=1}^{2m} a_i x_i^2, \quad 0 \neq a_i \in k$$
and consider the reflection
$$(x_1,x_2,\ldots,x_{2m}) \mapsto (-x_1,x_2,\ldots,x_{2m}).$$
This has discriminant $-1$ and induces the desired isomorphism.
\end{proof}
\begin{rema} \label{rema:weakvq}
We show that the components are stably birational without choosing
a reflection: Consider all hyperplanes $W \subset V$. 
Working over the dual space, we obtain a birational
$$
\OGr(m,q)_{\pm} \times V^{\vee} 
\stackrel{\sim}{\dashrightarrow} 
\OGr(m-1,q|W) \times V^{\vee}.$$
Here we regard $q|W$ as a quadratic form over 
the function field of $V^{\vee}$.  
We conclude that the $\OGr(m,q)_{\pm}$ are
mutually stably birational.

Here is an argument valid for odd $m$:
$\OGr(m,q)_+ \times \OGr(m,q)_-$ is birational
to $\OGr(m,q)_+ \times \bA^{m(m-1)/2}$ and
$\bA^{m(m-1)/2} \times \OGr(m,q)_-$.
Thus the components are mutually stably birational.
\end{rema}

\begin{prop} \label{prop:pointtorat}
Let $(V,q)$ be a non-degenerate quadratic form 
over a field $k$ and $\ell$ a positive integer
such that $2\ell < \dim(V)$. If 
$\OGr(\ell,q)$ has a $k$-rational point 
then it is rational over $k$.
\end{prop}
\begin{proof}
Let $X=\{q=0\} \subset \bP(V)$ and 
$P=\bP(L) \subset X$ be the subspace
associated with the $k$-rational point
$[L] \in \OGr(\ell,q)$. Projection
from $P$ gives a dominant rational map
$$\pi:\OGr(\ell,q) \dashrightarrow \Gr(\ell,V/L).$$
Given $L'\subset V/L$ of dimension $\ell$, the induced
extension 
$$0 \ra L \ra E \ra L'\ra 0$$
has dimension $2\ell$. The restriction $Q=q|E$ 
is a quadratic form of rank $2\ell$, with tautological maximal isotropic subspace $L$.  
The generic fiber of $\pi$ is isomorphic to $\OGr(\ell,Q)_{\circ}$ 
over the function field $k(\Gr(\ell,V/L))$. 
Here $\circ$ denotes 
one component of the orthogonal Grassmannian --
the one parametrizing subspaces intersecting $L$
in a subspace of even dimension.
This contains $L$ if and only if $\ell$ is even.

The second part of Proposition~\ref{prop:linalg}
implies that $\OGr(\ell,q)$ is rational
over the function field of $\Gr(\ell,V/L)$,
hence rational over $k$.
\end{proof}
Subspaces of dimension $\ell$ contains
subspaces of smaller dimensions, thus we have:
\begin{coro} \label{coro:ratsmalldim}
Retain the notation of Proposition~\ref{prop:pointtorat}.
If $\OGr(\ell,q)$ has a $k$-rational point 
then $\OGr(\ell',q)$ is rational, or a disjoint union of two
rational components, for each $\ell' \le \ell$.
\end{coro}

\subsection*{Witt rings over fields}
Consider non-degenerate quadratic forms $(V,q)$ over $k$,
under the operations of direct sum and tensor product.  
Recall that a quadratic form is {\em hyperbolic} if it may be
expressed as a direct sum of two-dimensional vector spaces
with hyperbolic form $xy$. Recall fundamental results of 
Witt: Every non-degenerate form admits an expression
$$
(V,q) = (V_{an},q_{an}) \oplus_{\perp} H,
$$
where the first summand is {\em anisotropic} -- has no 
non-trivial isotropic subspace -- and the second summand is hyperbolic.
Moreover, the anisotropic summand is unique up to isomorphism;
indeed, if $r$ is the maximal dimension of an isotropic subspace
of $q$ then $H$ is a direct sum of $r$ hyperbolic planes.  
The results above (like Corollary~\ref{coro:ratsmalldim})
imply that {\em all} such
decompositions are parametrized by a rational variety or the 
disjoint union of two rational varieties.

For any non-degenerate quadratic form $(V,q)$, the direct sum
$$(V,q) \oplus_{\perp} (V,-q)$$
is hyperbolic, with the diagonal subspace isotropic. 
The tensor product of any form with a hyperbolic form is hyperbolic.
\begin{defi} \cite[\S 2]{EKMbook}
\label{defn:wittring}
The {\em Witt ring} $\Witt(k)$ is the collection of isomorphism
classes of non-degenerate
quadratic forms under direct sum and tensor product, modulo the 
ideal of hyperbolic forms. 
\end{defi}
Each non-zero element of $\Witt(k)$ is represented by an anisotropic quadratic form. 

The {\em augmentation ideal} $I \subset \Witt(k)$ consists of all
the quadratic forms of even dimension. The associated graded ring
$$\Witt(k)/I  \oplus I/I^2 \oplus I^2/I^3 \oplus \cdots $$
encodes many key invariants: The discriminant of an even-rank 
form is captured by $I/I^2$ and the Clifford invariant by $I^2/I^3$.

\section{Compactifications}
\label{sect:compactify}
Let $(V,q)$ be a non-degenerate quadratic form
over $k$ with $\dim(V)=2m>0$. 
The projectivization of $\SO(V,q) \subset \End(V)$ 
is the {\em projective orthogonal group}
$$\PSO(V,q) \subset \bP(\End(V)).$$
It is the adjoint form of the special orthogonal group.
We interpret the similitude character 
$$\lambda \in \Gamma(\cO_{\PSO(V,q)}(2)),$$
as it is a homogeneous quadratic form in the matrix
entries. 

\subsection*{Geometric normal forms}
We sketch the geometry of closure of the orthogonal 
group, assuming $k$ is algebraically closed. 
More formal descriptions will follow after
we introduce the wonderful compactification.

Since $q$ is hyperbolic, we may choose coordinates $$x_{1'},\ldots,x_{m'},x_{1''},\ldots,x_{m]]}$$ such that
$$q:=x_{1'}x_{1''}+\cdots+x_{m'}x_{m''}.$$
The maximal torus of $\bG_m \times \SO(V,q)$ acts by
$$(x_{1'},\ldots,x_{m'},x_{1''},\ldots,x_{m''}) \mapsto
(t_1x_{1'},\ldots,t_mx_{m'},\lambda t_1^{-1} x_{1''},\ldots, \lambda t_m^{-1} x_{m''}).$$
There is a natural homomorphism
$$\bG_m \times \SO(V,q) \ra \GO(V,q).$$
The Weyl group $\mathrm W(\mathsf D_m)$ acts on the maximal torus of
$\SO(V,q)$ via signed permutations, exchanging the indices
$\{1,\ldots,m\}$ and inverting various $t_i$. 

The eigenvectors of a generic semisimple element of $\SO(V,q)$ determines a
collection of $2^{m-1}$ decompositions into complementary isotropic subspaces
\begin{align*}
V &= W' \oplus W'',\\
   &W'=\{x_{1''}= \cdots = x_{m''}=0\}, \quad 
    W''=\{x_{1'}=\cdots = x_{m'}=0\}.
\end{align*}
Their orbits under $\mathrm W(\mathsf D_m)$ depends on the parity of $m$.
For odd $m$, $W'$ and $W''$ come from distinct varieties
$\OGr(m,q)_+$ and $\OGr(m,q)_-$, thus each decomposition admits
an action by the discriminant. The group $C_2^{m-1} \subset
\mathrm W(\mathsf D_m)$ acts simply transitively on the collection.
For even $m$, $W'$ and $W''$ come from the same component
of the variety of maximal isotropic subspaces, and 
$C_2^{m-1}$ has two orbits.

We consider non-zero limiting elements in the closure 
$$\PSO(V,q) \subset \oPSO(V,q) \subset \bP(\End(V));$$
these may be understood as where $\lambda =0$. 
Such elements are conjugate to
$$(x_{1'},\ldots,x_{m'},x_{1''},\ldots,x_{m''}) \mapsto
(t_1x_{1'},\ldots,t_jx_{j'},0, \ldots, 0),$$
for $j=1,\ldots,m$.
In other words, up to the action of the Weyl group 
and exchanging $x_{i'}$ and $x_{i''}$, these are where 
$$\lambda=t_{j+1}=\cdots=t_m=0$$
and
$$x_{(j+1)'}=\cdots = x_{m'}=x_{1''}=\cdots=x_{m''}=0.$$
These have rank $\le m$ and we describe those of rank $m$:

\begin{prop} \label{prop:Dm-1Dm}
The variety of rank-$m$ elements 
$$\{A \in \oPSO(V,q): \operatorname{rank}(A)=m \}$$
has two irreducible components, of dimension 
$$2m^2-m-1=\dim(\PSO(V,q))-1,$$ 
described as follows:
Generic elements $A$ correspond
to decompositions into isotropic subspaces
$$V=W' \oplus W''$$
invariant under $A$,
where $A|W'=0$ and $A|W''$ is invertible
(resp.~$A|W''=0$ and $A|W'$ is invertible).

If $m$ is odd then $W'$ and $W''$ come from distinct 
components $\OGr(m,q)_{\pm}$; when $m$ is even they come from
same component of the maximal isotropic Grassmannian. 
The resulting divisors are $\PGL_m$-bundles over
$$\OGr(m,q)_+ \times \OGr(m,q)_-,$$
when $m$ is odd, and over
$$\OGr(m,q)_+ \times \OGr(m,q)_+ \text{ and }
\OGr(m,q)_- \times \OGr(m,q)_-$$
when $m$ is even.
\end{prop}

We close with two remarks about the strata of rank $<m$:
\begin{itemize}
\item{they have dimension strictly smaller than the divisorial
strata of Proposition~\ref{prop:Dm-1Dm};}
\item{for matrices $A$ of rank $<m$ the subspaces 
$W'$ and $W''$ are not uniquely determined by the matrix.}
\end{itemize}

\subsection*{Wonderful constructions}
We consider compactifications of $\PSO(V,q)$, bi-equivariant
under the action: 
\begin{align*}
\PSO(V,q)^2 \times \PSO(V,q) & \ra \PSO(V,q) \\
((g_1,g_2),g) & \mapsto g_1gg_2^{-1}.
\end{align*}
These were pioneered by De Concini-Procesi \cite{DP}
and subsequently extended to general fields \cite{Str,DS}.
Details on their Picard groups, from
the perspective of spherical varieties,
may be found in \cite{Brion} and \cite{BrionTotal}.

Consider the simply-connected cover of $\PSO(V,q)$ 
over an algebraic closure and fix positive and simple
roots. Let $\blambda$ denote a 
dominant weight with associated representation 
$V_{\blambda}$ and endomorphisms $\End(V_{\blambda})$. 
Let 
$$X_{\blambda}(V,q)\subset \bP(\End(V_{\blambda}))$$ 
denote the closure of the identity under the 
$\PSO(V,q) \times \PSO(V,q)$ action on 
$$\End(V_{\blambda})=V_{\blambda} \otimes V_{\blambda}^{\vee}.$$ 
When $\blambda$ is regular -- a positive sum of the
fundamental weights, in the interior of the Weyl chamber -- we obtain the {\em wonderful compactification} \cite[3.4]{DP}
$$\overline{X}(V,q):=X_{\blambda}(V,q) \supset\PSO(V,q).$$ 
For weights $\blambda'$ 
along the boundary of the Weyl chamber, we obtain bi-equivariant morphisms
$$X_{\blambda}(V,q) \ra X_{\blambda'}(V,q).$$
For $V_{\blambda'}=V$, the
$2m$-dimension standard representation,
the distinguished orbit in $\bP(\End(V))$ is the
projective similitude group. 
This yields an explicit iterated blowup of its closure
$$\beta: \overline{X}(V,q)\stackrel{\sim}{\longrightarrow} X_{\blambda'}(V,q)=\oPSO(V,q) \subset \bP(\End(V)).$$  

We record standard facts about $\oX(V,q)$ over
algebraically closed fields \cite[3.1]{DP}:
\begin{itemize}
\item{$\oX(V,q)$ is smooth and projective;}
\item{the boundary $\oX(V,q) \setminus \PSO(V,q)$ is 
a normal crossings divisor with irreducible components $D_i$ indexed by the simple roots and strata/orbits indexed by subsets of those roots;}
\item{the minimal stratum $Y \simeq (\PSO(V,q)/B)^2$, 
the product of complete isotropic flag varieties.}
\end{itemize}
Furthermore, analyzing the restriction homomorphism
$$\Pic(\oX(V,q)) \ra \Pic(\PSO(V,q)/B)^2,$$
we have \cite[{\S}8]{DP} and \cite[2.1.3]{BrionTotal}:
\begin{itemize}
\item{$\Pic(\oX(V,q))$ is isomorphic to the character group $\Lambda$ of the simply-connected cover of $\PSO(V,q)$, and is realized as the graph of an explicit isomorphism $$\Pic(\PSO(V,q)/B)\stackrel{\sim}{\ra}\Pic(\PSO(V,q)/B).$$}
\item{The Picard group is freely generated by the ``color'' divisors
$C_1,\ldots,C_m$, closures of codimension-one components of $$\PSO(V,q) \setminus (B^-\times_T B),$$ 
where $B^-$ is the ``opposite'' Borel with $B^-\cap B=T$ the maximal torus. The color divisors are indexed by fundamental
weights $\omega_1,\ldots,\omega_m$ of the simply-connected cover of $\PSO(V,q)$ \cite[2.2.4]{BrionTotal}.}
\item{Nef divisors are non-negative linear
combinations of $C_1,\ldots,C_m$ \cite[2.6]{Brion}.}
\item{The cone (but not necessarily the monoid) of effective divisors is generated by
$D_1,\ldots,D_m$, realized as simple roots
\cite[2.3.5]{BrionTotal}.}
\end{itemize}

Given a weight $\bmu\in \Lambda$, let $L_{\bmu}$ be the associated line bundle on $\oX(V,q)$. 
Brion \cite[3.2.4]{BrionTotal} computes the total coordinate ring 
$$
R(\oX(V,q))=\oplus_{\bmu\in \Lambda} \Gamma(\oX(V,q),L_{\bmu}).
$$
For our purposes, it suffices to mention that for a dominant weight $\blambda$, there is distinguished summand
$$\End(V_{\blambda}) \subset \Gamma(\oX(V,q),L_{\blambda})$$
associated with the morphism 
$$\oX(V,q) \ra X_{\blambda}(V,q) \subset \bP(\End(V_{\blambda})).$$
In particular, we can compute the image of $\beta^*$
on the Picard group.

We explore the geometry of the boundary divisors under 
$\beta$. The simple roots $\alpha_1,\ldots,\alpha_{m-1},\alpha_m$
may be ordered so that the first $m-2$ are the strand
of the Dynkin diagram and $\alpha_{m-1},\alpha_m$ are the
two nodes at the end, exchanged by the involution of the
diagram. (This assignment is ambiguous for $m=4$ due to triality.)

\begin{prop} \label{prop:excmulttwo}
The divisors $D_1,\ldots,D_{m-2}$ are exceptional for 
$\beta$. 
The divisors $D_{m-1}$ and $D_m$ correspond to projections
onto pairs of maximal isotropic subspaces 
$$\OGr(m,q)_+ \times \OGr(m,q)_- \text{ (odd $m$)} \ \text{ or } \ \OGr(m,q)_{\pm}^2 \text{ (even $m$)}$$ 
composed with isomorphisms of one of those subspaces.
The pull back of the similitude divisor $\{\lambda=0\}$
under $\beta$ has class
$$2(D_1+\cdots+D_{m-2})+D_{m-1}+D_m,$$
so the exceptional divisors all appear with multiplicity two.
\end{prop}
\begin{proof} The description of $D_{m-1}$ and $D_m$ are
given in Proposition~\ref{prop:Dm-1Dm}. The remaining
strata $D_1,\ldots,D_{m-2}$ correspond to lower-rank matrices.

In light of the descriptions of effective and nef divisors above,
the last assertion is a computation with root systems. We use the notation
of \cite[19.2]{FulHar} with fundamental weights
$$\omega_i = \begin{cases} L_1+\cdots+L_i & i\le m-2 \\
                        (L_1+\cdots+L_{m-1}-L_m)/2 & i=m-1\\
                        (L_1+\cdots+L_{m-1}+L_m)/2 & i=m
                        \end{cases}
$$
and simple roots
$$\alpha_i= L_i-L_{i+1}, \ i=1,\ldots,m-1, \quad
\alpha_m=L_{m-1}+L_m.$$
These correspond to the divisors $C_i$ and $D_i$
respectively. Thus we find
\begin{align*}
\beta^*\{\lambda = 0\} = 2L_1 =
     2(L_1-L_2)+ \cdots + 
        2(L_{m-2}-L_{m-1}) & \\
         + (L_{m-1}-L_m) + (L_{m-1}+L_m) & \\
        = 2D_1+ \cdots + 2D_{m-2} + D_{m-1} + D_m. &
\end{align*}
\end{proof}

\begin{rema}[Extension to orthogonal groups]
\label{rema:EOG}
Recall that the groups $\OO(V,q)$, $\GO(V,q)$, and the
projectivization $\PGO(V,q)$
have two connected components, distinguished by the
discriminant character $\disc$. 
Concretely, when $q$ is hyperbolic we can express
$$\OO(V,q) = \SO(V,q) \rtimes \left<\iota\right>$$
where $\iota$ is an involution exchanging $\{x_{i+},x_{i-}\}$
for some index $i$.  
The construction of $X_{\blambda}(V,q)$ yields compactifications of $\PGO(V,q)$ with two components isomorphic to $X_{\blambda}(V,q)$.
Thus $\PGO(V,q)$ admits a wonderful compactification 
consisting of two copies of $\oX(V,q)$. 
These are all equivariant for the $\PGO(V,q)\times \PGO(V,q)$
action 
$$((g_1,g_2),g) \mapsto g_1gg_2^{-1}.$$
\end{rema}

\section{Equivariant quadratic forms}
\label{sect:equiv-quad}

From now on, $k$ is algebraically closed of characteristic zero
unless specified otherwise.

Let $G$ be a finite group and consider a projective quadric hypersurface $X \subset \bP^n$ 
admitting a regular action of $G$.  
Using the extension
$$ 1 \rightarrow \bG_m \rightarrow \GL_{n+1} \rightarrow \PGL_{n+1} \rightarrow 1$$
we extract the following data:
\begin{itemize}
\item{an extension 
\begin{equation} \label{eqn:Amitsur} 1 \rightarrow \bG_m \rightarrow \wG \rightarrow G \rightarrow 1; \end{equation}}
\item{a linear representation of dimension $n+1$
$$\rho:\wG \rightarrow \GL(V)$$ 
inducing the $G$-action on $\bP^n=\bP(V)$;}
\item{a $G$-quadratic form, i.e., a $G$-invariant element $q \in \Sym^2(V^{\vee})\otimes L$, where $L$ is a one-dimensional
representation with associated character $\lambda:\wG\ra \bG_m$, such that
$$X= \{q=0\}, \quad
q(\rho(g)\cdot v)=\lambda(g) q(v).
$$
}
\end{itemize}
The obstruction to the extension (\ref{eqn:Amitsur}) splitting is a class
$$\alpha \in \rH^2(G,\bG_m)$$
known as the {\em Amitsur invariant}. If $\alpha\neq 0$ then $\bP^n$ and its $G$-subvarieties are not stably linearizable;
see \cite[\S 2]{BP}, \cite[App.A]{BCDP}, \cite[\S 2]{HTPAMQ}, and \cite[\S 3.5]{HTNagoya} for more context.
From now on, we impose the vanishing of this invariant:
\begin{defi}
A {\em $G$-quadric} is a quadric hypersurface $X\subset \bP(V)$,
where $\rho:G \ra \GL(V)$ is a linear representation and 
$X$ is invariant under the action of $G$.
\end{defi}

The $G$-quadratic form $q$ is equivalent to a $G$-linear isomorphism
$$s:V \lra V^{\vee}\otimes L,$$
symmetric in that it is unchanged on applying $\Hom(-,L)$,
via the functional relation
$$q(v) = s(v)\cdot v.$$
Thus these correspond to $G$-representations in the group of similitudes.
Write $\lambda:G \rightarrow \bG_m$ for the character associated with $L$.
\begin{exam}
Let $G=C_n=\left<\gamma\right> $ be a cyclic group of order $n$, and $V=\left<v_1,v_2\right>=\{xv_1+yv_2\}$ a
representation
$$\gamma\cdot x=\zeta x, \quad \gamma\cdot y=\zeta^a y, \quad \zeta=\exp(2\pi i/n),$$
with quadratic form $q=xy$ with $\lambda(\gamma)=\zeta^{a+1}$.  
\end{exam}

\begin{assu}
We will usually assume that $X$ is {\em smooth}, or equivalently,
$s:V \stackrel{\sim}{\lra} V^{\vee}\otimes L$
or $q$ is {\em non-degenerate}.  
\end{assu}

\begin{rema}
For any character $\chi:G \rightarrow \bG_m$, the representations $\rho$ and $\rho\otimes \chi$
yield the same projective representation $G \ra \PGL(V)$. 
Once the $G$-action on $X\subset \bP^n=\bP(V)$ is specified, $\rho$ is uniquely determined up to a character. 
\end{rema}

We define basic invariants of smooth $G$-quadrics $X\subset \bP(V)$. The {\em rank} of $X$ is $\dim(V)$,
which has the same parity as $\dim(X)$.

\subsection*{Similitude class}
Consider $V\otimes N$, where $N$ is a one-dimensional representation with character $\nu$. 
The resulting $(V\otimes N,q)$ with
$$q \in \Sym^2(V^{\vee})\otimes L = 
\Sym^2((V\otimes N)^{\vee})\otimes (N^2 \otimes L),$$
has character $\nu^{2} \lambda$.

The {\em similitude class} of $X=\{q=0\}, q\in \Sym^2(V^{\vee})\otimes L$ is the element
\begin{align*}
[\lambda] \in &\cok\left(\Hom(G,\bG_m) \stackrel{\cdot 2}{\lra} \Hom(G,\bG_m)\right) \\
             & = \ker\left( \rH^2(G,\mu_2)\lra \rH^2(G,\bG_m)\right),
\end{align*}
where $\lambda$ is the character associated with $L$.

The class vanishes if and only if there exists a representation $\rho:G \ra \GL(V)$ and an embedding $X\hookrightarrow \bP(V)$ such that
$X$ is defined by a $G$-invariant quadratic form $q \in \Sym^2(V^{\vee})$.

\subsection*{Discriminant}
Write $X=\{q=0\} \subset \bP(V)$, for $q\in \Sym^2(V^{\vee}) \otimes L$.
Taking determinants gives an isomorphism
$$\det(V)^2 \simeq L^{\dim(V)}.$$
When $\rank(X)$ is even, we obtain a character
$$\left(\det(V) \otimes L^{-\dim(V)/2}\right): G \ra \mu_2$$
and an invariant 
$$
\disc(X):=[\det(V) \otimes L^{-\dim(V)/2}] \in \Hom(G,\mu_2),
$$
the {\em $G$-discriminant} of $X$.  
The $G$-discriminant is independent of the choice of $V$ and $L$, i.e., unchanged under
$$V \mapsto W\otimes N, \quad L \mapsto M\otimes N^2.$$ 

\begin{rema} \label{rema:discmaxiso}
When $\dim(V)=2m$ and $q$ is non-degenerate then
the discriminant encodes the action of $G$ on
$\OGr(m,q)_{\pm}$ cf.~\cite[\S 85]{EKMbook}.
\end{rema}

\subsection*{Clifford invariant}
Let $X=\{q=0\} \subset \bP(V)$ be a smooth $G$-quadric with $\dim(V)=2m$ and trivial discriminant.
The embeddings
$$\OGr(m,q)_{\pm} \hookrightarrow \bP(S_{\pm})$$
induce projective representations of $G$.  
Let $\alpha_{\pm}\in \rH^2(G,\bG_m)[2]$ denote their Amitsur invariants, the obstructions
to lifting $S_{\pm}$ to linear representation of $G$, the {\em $G$-Clifford invariants}.

\begin{rema} Proposition~\ref{prop:spinAmitsur} shows that 
the classes $\alpha_{\pm}$ are related to the classical Clifford invariant of \cite[\S 14]{EKMbook}, defined as the 
Brauer class of $C^+(q)\simeq \End(S_+) \otimes \End(S_-)$ (see \cite[13.9]{EKMbook}). 
\end{rema}

\begin{rema}
Parimala and Srinivas \cite{ParSri} define 
cycle classes for Brauer algebras with 
involution that yield the
Clifford invariant as a special case.
There is even a refinement
with values in $\rH^2(G,\mu_2)$.
Such refinements do not exist over $I_2(X)$
-- quadratic forms of even rank and trivial 
discriminant -- for a general base $X$.   
\end{rema}

Let $X \subset \bP(V)$ be a smooth $G$-quadric with 
$\dim(V)=2m+1$. 

\begin{exam}
We give an example of a quadric threefold that is linearizable but has non-trivial Clifford invariant. 
Consider the dihedral group 
$$
\left<\sigma,\tau: \sigma^4=\tau^2=1, \tau\sigma\tau^{-1}=\sigma^3\right>
$$
with representation $V$
$$ \sigma \mapsto \left( \begin{matrix} i & 0 \\
                                        0 & i^3 
                            \end{matrix} \right), \quad 
    \tau \mapsto \left( \begin{matrix} 0 & 1 \\ 1 & 0 
                            \end{matrix} \right).
                            $$
This gives a projective representation of 
$$
C_2\times C_2 =  \left<\sigma, \tau: \sigma^2=\tau^2=1\right>.
$$
Let $C$ be the corresponding conic, realized in
$\bP(\Sym^2(V))$, with a non-trivial Amitsur invariant. Consider $\bP(\Sym^2(V) \oplus \mathbf{1})$ which is birational to a quadric threefold $X$ blown up at a fixed point $p$
$$ \Bl_C(\bP(\Sym^2(V) \oplus \mathbf{1})) \simeq
\Bl_p(X).$$
We blow down the proper transform of $\bP(\Sym^2(V))$
to get $p$. Thus $X=\{q=0\}$ for some $C_2\times C_2$-invariant quadratic form.  

The variety $\OGr(2,q)$ is a Brauer-Severi threefold, birational to a 
$\bP^2$-bundle over $C$. Thus Clifford invariants can
be non-trivial for quadrics that are linearizable.  
\end{exam}


\section{Invariant isotropic subspaces}
\label{sect:inv}

\subsection*{Definitions and basic properties}

Let $X=\{q=0\} \subset \bP(V)$ be a $G$-quadric, where $V$ is a linear representation of $G$.
Write $L$ for the one-dimensional representation associated to $q$ and
$$s:V \stackrel{\sim}{\lra} V^{\vee}\otimes L$$
the associated symmetric isomorphism of $G$-representations.
Let $W \subset V$ be a $G$-stable isotropic subspace.

\begin{exam}
The singular locus of $G$-invariant quadric hypersurface $X \subset \bP(V)$ is the
projectivization of an isotropic subspace.  
\end{exam}

\begin{defi}
A $G$-quadric $X=\{q=0\}\subset \bP(V)$ is {\em anisotropic} if 
there exists no sub-representation $0\neq W \subset V$ such that $q|W=0$. 
\end{defi}

We consider examples for regular representations:
\begin{exam}
Suppose that $G$ is an abelian $2$-group that is not $2$-elementary,
i.e., $2G \neq 0$. Then $V:=k[G]$ is isotropic.

We may as well assume that $G=\left<g\right>$ is cyclic of order
$2^n\ge 4$ with quadratic form
$$q=x_0^2+x_1x_{2^n-1}+\cdots + x_{2^{n-1}-1}x_{2^{n-1}+1}
+x_{2^{n-1}}^2,
$$
where the $x_i$ are characters of $G$.  
The subspace 
$$
W=\{x_0=x_1=\ldots = x_{2^{n-1}}=0\}
$$
is isotropic.  
\end{exam}

\begin{prop} \label{prop:getperp}
Retain the notation introduced above and assume that $W\subset V$ is irreducible and isotropic.
Then there exists an isotropic irreducible $W'\subset V$ such that the induced pairing on $W \times W'$
is non-degenerate. Moreover $W' \simeq W^{\vee}\otimes L.$
\end{prop}
Our proof does not require that $k$ be algebraically closed.
\begin{proof}
Consider the quotient 
$V^{\vee} \twoheadrightarrow W^{\vee}$
and choose a $G$-equivariant splitting, yielding a copy of $W^{\vee}\otimes L \subset V^{\vee}\otimes L$.  
Consider $W'':=s^{-1}(W^{\vee}\otimes L)$ and the induced $L$-valued bilinear form on 
$W \times W'' \subset V$. We already know this is isotropic on the first factor and
induces the natural $L$-valued bilinear pairing on $W \times W''$. It remains to compute this in
$W''$; we are done if it is isotropic.

Suppose $W''$ is {\em not} isotropic. Then, by Schur's Lemma, we would have a symmetric isomorphism
$$W'' \stackrel{\sim}{\lra} (W'')^{\vee}\otimes L,$$
or equivalently, a symmetric
$$t:W^{\vee} \otimes L \stackrel{\sim}{\lra} W.$$
Schur's Lemma guarantees this is the unique non-trivial homomorphism between these representations, up to scalars.  
Thus we have
$$W \oplus W'' \simeq W \oplus W.$$
with quadratic form
$$q(w_1,w_2) = c_1[t^{-1}(w_1)\cdot w_2 + t^{-1}(w_2)\cdot w_1] + c_2 [t^{-1}(w_2)\cdot w_2], \ w_1,w_2 \in W,$$
for some nonzero constants $c_1,c_2$.  Applying row-and-column operations, we obtain
$$W' \simeq W \hookrightarrow W \oplus W''$$
that is isotropic and pairs non-degenerately with the first factor.  
\end{proof}

\begin{defi} \label{defi:hyp}
Fix a character $\lambda:G \ra \bG_m$ with associated representation $L$. Given a $G$-representation $W$,
the {\em hyperbolic lattice}
$$H_W(\lambda) = W \oplus (W^{\vee}\otimes L),$$
with bilinear form 
$$b(w_1,f_1;w_2,f_2)=\frac{1}{2}(f_1(w_2)+ f_2(w_1)).$$ 
This is non-degenerate and the summands $W$ and $W^{\vee}\otimes L$
are isotropic subspaces.  
\end{defi}
Note that any hyperbolic lattice is a direct sum of hyperbolic lattices arising from irreducible representations.

\subsection*{Classification results}
Here we extend the analysis of \cite[\S 3]{HTEGLDQ}.
The base field $k$ is algebraically closed.  
\begin{prop} \label{prop:decomp}
Consider a smooth $G$-quadric
$$X=\{q=0 \} \subset \bP(V)$$
with associated symmetric
$s:V \stackrel{\sim}{\lra} V^{\vee} \otimes L$
with character $\lambda: G \ra \bG_m$.  
The irreducible representations $W \subset V$ are 
of three types:
\begin{enumerate}
\item{$W$ has a symmetric homomorphism
$W\stackrel{\sim}{\rightarrow} W^{\vee} \otimes L$.}
\item{$W$ has a skew-symmetric
$W\stackrel{\sim}{\rightarrow} W^{\vee} \otimes L$.}
\item{$W$ admits no such isomorphism.}
\end{enumerate}
The isotypic component $W^m \subset V$, where
$m$ is the multiplicity of $W$ in $V$, 
in each case takes the form:
\begin{enumerate}
\item{$q|W^m$ is non-degenerate and isomorphic to
$H_W(\lambda)^n$ if $m=2n$ and $H_W(\lambda)^n \oplus W$
if $m=2n+1$;}
\item{$m=2n$ and $q|W^{2n} \simeq H_W(\lambda)^n$;}
\item{$W^{\vee}\otimes L$ also appears in $V$ with 
multiplicity $m$ and
$$q|(W^m \oplus (W^{\vee}\otimes L)^m) \simeq 
H_W(\lambda)^m.$$}
\end{enumerate}
Moreover, these decompositions are unique up to 
automorphisms of $q$.  
\end{prop}
\begin{rema}
This should be understood as an equivariant 
analog of the Witt decomposition theorems. 
While the hyperbolic summands are not canonical they
are determined up to automorphisms.
The trichotomy of irreducbile orthogonal representations goes back to Frobenius and Schur \cite[p.~58]{Isaacs}.  
\end{rema}
\begin{proof}
The trichotomy -- that the three cases are disjoint --
follows from Schur's Lemma. The proof of Proposition~\ref{prop:getperp} guarantees the
existence of an irreducible summand pairing 
non-degenerately with $W$.  In the first case,
this is isomorphic to $W$; however,
observe that $W^{\oplus 2}$ always has an 
isotropic copy of $W$. 
Any sum
$$(W,r) \oplus (W, cr), \quad c\in k, r = q|W,$$
admits an isotropic subspace, i.e.,~the image of 
$w \mapsto (\sqrt{-c} \cdot w,w).$
For the second case,
$W$ is necessarily isotropic and the $q|W^m$ 
is an $m\times m$ skew-symmetric matrix $(a_{ij})$,
where the $a_{ij}$ are the skew pairings 
$\wedge^2 W \rightarrow L$.
For the parity of $m$, observe that skew-symmetric matrices
always have even rank. The reduction theory of 
symplectic forms implies that the matrix is equivalent to
$$\left( \begin{matrix} 0 & Ia \\
            -Ia & 0 \end{matrix} \right)$$
where the entries are an $n\times n$ identity matrix $I$
times a fixed $a:\wedge^2 W\stackrel{\sim}{\ra}L$. 
This is symmetric once we take into account that the 
isomorphisms are all skew!
In the third case, $W$ is
also isotropic and the claim follows 
directly from Proposition~\ref{prop:getperp}.
\end{proof}

We obtain the following consequences:
\begin{coro} \label{coro:anihyp}
Let $X=\{q=0\} \subset \bP(V)$ be a smooth $G$-quadric.  
We may express
$$V = V_a \oplus_{\perp} V_h$$
where $q|V_a$ is anisotropic and $q|V_h$ is hyperbolic.  
Write
$$X_a := \{q|V_a=0\} = X \cap \bP(V_a)$$
for the anisotropic section of $X$, which is 
determined up to isomorphism.
\end{coro}
Our next result extends \cite[Prop.~3.7]{HTEGLDQ}:
\begin{coro} \label{coro:struct}
Fix a character $\lambda:G \ra \bG_m$ with associated representation $L$.
Each anisotropic $X=\{q=0\} \subset \bP(V)$ with 
character $\lambda$
expressed uniquely as a sum
$$(V,q) \simeq \oplus_{i=1}^n (W_i,r_i)$$
where the $W_i$ are distinct irreducible representations of $G$ and
$r_i$ encodes a symmetric
$W_i \stackrel{\sim}{\lra} W_i^{\vee} \otimes L.$
\end{coro}

\begin{exam}[Regular representations]
Given a finite group $G$, let $V=k[G]$ be the regular
representations with coordinates 
$$\{x_g, g\in G\}, \quad x_g(g')=\delta_{g,g'}.$$
Consider the $G$-invariant non-degenerate quadratic form
$$q=\sum_{g\in G} x_g^2.$$
Its discriminant is the sign character for the permutation
representation of $G$ on itself.  

Now let $W$ be an irreducible representation of $G$ 
and consider the isotypic component
$$W^{\dim(W)} \subset V.$$
We apply Proposition~\ref{prop:decomp} and its corollaries:
The anisotropic part $V_a \subset V$ is the 
sum of irreducible self-dual representations of odd dimension,
each with multiplicity one. When $G$ is a $2$-group --
so all the irreducible representations satisfy
$\dim(W )=2^n$ -- 
$V_a$ is the direct sum of the characters $G \ra \mu_2$.
\end{exam}

\subsection*{Birational implications}

We recall birational properties 
of quadrics with invariant isotropic subspaces, following \cite[\S 3]{HTEGLDQ} and especially Theorem~3.5.
\begin{rema} \label{rema:isostab}
Suppose that $X\subset \bP(V)$ is a smooth $G$-quadric
of positive dimension.
\begin{itemize}
\item{If $X$ has a fixed point then it is linearizable;}
\item{If $X$ has an invariant isotropic subspace then the $G$-action on $X$
is stably linearizable.}
\end{itemize}
More precisely, suppose an invariant isotropic subspace $W$
yields
$$V = H_W(\lambda) \oplus H_W(\lambda)^{\perp}  = W \oplus (W^{\vee} \otimes L) \oplus H_W(\lambda)^{\perp}$$ 
via Proposition~\ref{prop:getperp} and Definition~\ref{defi:hyp}. If the $G$-action on
$$
\bP(W^{\vee} \oplus H_W(\lambda)^{\perp})
$$ 
is 
generically free then $X$ is linearizable.  
\end{rema}

\subsection*{Splitting and varieties 
of linear subspaces}
Let $V$ be a representation of a finite 
group $G$, $L$ a character of $G$, and
$X = \{q=0 \} \subset \bP(V)$
a smooth $G$-quadric arising from a 
symmetric $s:V \stackrel{\sim}{\ra} V^{\vee} \otimes L$. 

Let $\OGr(\ell, q)$ denote the 
variety of isotropic subspaces 
of dimension $\ell$ for $q$, 
with the induced $G$-action.
Remark~\ref{rema:weakvq} yields:
\begin{prop} \label{prop:maxstabrat}
Suppose that $\dim(V)=2m$ with trivial discriminant, so
that $\OGr(m,q)_{\pm}$ both have $G$-actions. Then they
are equivariantly stably birational.
\end{prop}
This is weaker than Proposition~\ref{prop:vq} but
all we can expect in the equivariant context:
\begin{exam}
Let $G=C_5$ act on $\bP^1_{u,v}$ and $\bP^1_{x,y}$
\begin{align*}
[u,v] &\mapsto [\zeta u, \zeta^4 v] \\
[x,y] &\mapsto [\zeta^2 x, \zeta^3 y]
\end{align*}
where $\zeta$ is a primitive fifth root of unity.
These are not isomorphic as $G$-varieties.
Consider the product
$$X = \bP^1_{u,v} \times \bP^1_{x,y} \subset \bP(V)$$
where $V$ is the tensor product of the two-dimensional representations
above, with weights $\{\zeta^3,\zeta^4,\zeta,\zeta^2\}$.
The variety of lines
$$\OGr(2,q) = \bP^1_{u,v}\sqcup \bP^1_{x,y}$$
hence the $\OGr(2,q)_{\pm}$ are not isomorphic $C_5$-varieties. They are stably birational by the No Name Lemma.
\end{exam}

Our next result generalizes Corollary~\ref{coro:ratsmalldim}:
\begin{prop} \label{prop:Wittdimred}
Let $X=\{q=0\}\subset \bP(V)$ be a $G$-quadric and assume $G$
acts generically freely on $\OGr(\ell,q)$. For each $\ell'<\ell$, 
$\OGr(\ell',q) \times \OGr(\ell,q)$
is stably linearizable over $\OGr(\ell,q)$
\end{prop}

Our reasoning is similar to \cite[Sec.~2.2]{lena}. 

\begin{proof}
Let $\cS \ra \OGr(\ell,q)$ denote the
universal isotropic sub-bundle and consider $\ell'$-dimensional
subspaces of the quotient
$$\gamma:\Gr(\ell',V\otimes \cO_{\OGr(\ell,q)}/\cS) \ra \OGr(\ell,q),$$
with universal bundle $\cL'$. We obtain an associated rank-$(\ell'+\ell)$ extension
$$0 \ra \gamma^*\cS \ra \cE \ra \cL' \ra 0$$
contained in the trivial bundle with fiber $V$. Let $Q=q|E$
a quadratic form degenerate along $\cK:=\gamma^*\cS \cap (\cL')^{\perp}$, a 
rank-$(\ell-\ell')$ subbundle. Thus we have a non-degenerate quadratic form $Q'$
on $\cE/\cK$ -- of rank $2\ell'$ -- with a tautological maximal isotropic
subspace $\gamma^*\cS/\cK$. Thus the space of all such maximal isotropic
subspaces is stably linearizable over $\Gr(\ell',V\otimes \cO/\cS)$ by Proposition~\ref{prop:linalg} (see also the
proof of Proposition~\ref{prop:pointtorat}).
Fixing an $\ell'$-dimensional isotropic subspace over $Q'$, the $\ell'$-dimensional isotropic subspaces of $Q$ mapping onto that subspace are parametrized by $\Gr(\ell',\ell)$.

Now $\Gr(\ell',q)\times \Gr(\ell,q)$ can be described -- equivariant birationally over $\Gr(\ell,q)$ -- as a tower of bundles
$$
\xymatrix{
\OGr(\ell',q) \times \OGr(\ell,q)
\ar@{-->}[rr]^{\sim} \ar[dr]& &
\cA \ar[dl] \\
    & \OGr(\ell,q) & 
}
$$
where $\cA \ra \OGr(\ell,q)$ factors as follows:
a $\Gr(\ell',\ell)$ bundle, over a split $\OGr(\ell',2\ell')$ bundle, over a
$\Gr(\ell',\dim(V)-\ell)$ bundle. It follows that $\OGr(\ell',q) \times \OGr(\ell,q)$
is stably linearizable over $\OGr(\ell,q)$.
\end{proof}

\begin{rema}
Here is an alternative proof: After stabilization, an $\ell$-dimensional subspace
guarantees an $\ell'$-dimensional subspace for $\ell'<\ell$. Over fields, 
rationality of $\OGr(\ell',q)$ follows from Proposition~\ref{prop:pointtorat}. 
The equivariant statement follows then from Proposition~\ref{prop:dr}. \end{rema}

The same reasoning yields the following equivariant version
of Proposition~\ref{prop:pointtorat}:
\begin{coro} \label{coro:fixedtorat}
Suppose that $X=\{q=0\} \subset \bP(V)$ is a $G$-quadric
and $\OGr(\ell,q)$ admits a $G$-fixed point or even an
equivariant morphism from a stably linearizable $G$-variety. Then
$\OGr(\ell',q)$ is stably linearizable for each $\ell' \le \ell$.
\end{coro}

\subsection*{Relations between quadrics}
We note one last classical fact:
\begin{prop} \label{prop:getstabrat}
Let $X$ and $Y$ be smooth $G$-quadrics of positive
dimensions $d$ and $e$. Assume there
exists a $G$-equivariant rational map
$ X \dashrightarrow Y$
and the action on $X$ is generically free.
Then there exists a $G$-equivariant birational
map
$$Y \times X \stackrel{\sim}{\dashrightarrow} \bP^e \times X.$$

Suppose the actions on $X$ and $Y$ are both generically
free and we have equivariant rational maps in both directions
\begin{equation} X \dashrightarrow Y, \quad 
Y \dashrightarrow X. \label{eqn:tworationalmaps}
\end{equation}
Then we obtain a $G$-equivariant stable birational equivalence
$$
Y \times \bP^d \stackrel{\sim}{\dashrightarrow} \bP^e \times X.$$
This conclusion holds also under a weakening of (\ref{eqn:tworationalmaps}),
i.e., that there are rational maps
$$X \times \bP^n \dashrightarrow Y, 
\quad Y \times \bP^m \dashrightarrow X.$$
\end{prop}
\begin{proof}
The rational map from $X$ to $Y$ means that projection
$$Y \times X \ra X$$
has a rational section. Now a quadric bundle with a section
is birational over $X$ to a projective bundle on $X$.
This bundle is birational to the trivial bundle 
$\bP^e \times X$ by the No-Name
Lemma, as the $G$-action on $X$ is generically free.
Applying this reasoning twice gives the stable birational
equivalence. Repeating this argument using the rational
maps from the products gives
$$Y\times \bP^d \times \bP^m \times \bP^n \stackrel{\sim}{\dashrightarrow}
X \times \bP^e \times \bP^m \times \bP^n$$
and our final assertion.
\end{proof}

\begin{prop} \label{prop:whyPfister}
Let $X\subset \bP(V)$ be a smooth $G$-quadric and 
$Y \subset X$ is a $G$-invariant smooth linear section
of codimension $c$. Assume that
\begin{itemize}
\item{the $G$-action on $Y$ is generically free;}
\item{for some $\ell>c$, there is an equivariant rational
map $X \dashrightarrow \OGr(\ell,q)$, or even an equivariant
$X \times \bP^n \dashrightarrow \OGr(\ell,q)$ for some $n$.}
\end{itemize}
Then $X$ and $Y$ are $G$-equivariantly stably birational.
\end{prop}
\begin{proof}
As the action is generically free on $Y$, the same
is true for $X$. We follow the argument for Proposition~\ref{prop:getstabrat}; clearly, $X\times Y$
is stably birational to $Y$.
A dimension count
shows that projective subspaces parametrized by
$\OGr(\ell,q)$ intersect $Y$
non-trivially, in isotropic subspaces. 
Remark~\ref{rema:isostab} gives that $X\times Y$ 
is equivariantly stably birational to $X$ as well. 
Thus $X$ and $Y$ are stably birational to each other.
\end{proof}
\begin{rema}
Proposition~\ref{prop:whyPfister} only applies
when $2\dim(Y)\ge \dim(X)$. Otherwise, we cannot
find the required isotropic subspaces in $X$!
\end{rema}

\subsection*{Reductions to 2-Sylow subgroups}
Results of Duncan and Reichstein \cite[Th.~10.2]{DR}
(see also \cite[Sec.~5.3]{HTEGLDQ}) imply that a $G$-quadric $X$ is stably linearizable
if and only if it is stably linearizable for any
$2$-Sylow subgroup $G_2\subset G$.  
Similar results apply for stable birational equivalence:
\begin{prop}
Let $X$ and $Y$ be smooth quadrics with 
generically free
actions of a finite group $G$.  
These are stably birational over $G$ iff
they are stably birational over a $2$-Sylow subgroup.
\end{prop}
In analyzing equivariant stable birational equivalence of quadrics, we need only consider
actions by $2$-groups.  
\begin{proof}
The forward implication is immediate; we focus on the reverse direction. 

Consider the stable birational equivalence over $G_2$
$$X \times V \stackrel{\sim}{\dashrightarrow} Y \times W.$$
This yields $G_2$-equivariant rational maps
$$X \dashrightarrow Y, \quad Y \dashrightarrow X.$$ 
The projections
$$X\times Y \rightarrow X, \quad X\times Y\rightarrow Y$$
therefore have $G_2$-equivariant rational sections
$\Sigma_X$ and $\Sigma_Y$.  
Let $G\cdot \Sigma_X$ and $G\cdot \Sigma_Y$ denote the closures of
the $G$-orbits of these rational sections. These have odd degree dividing
$[G:G_2]$ over $X$ and $Y$ respectively.

By Proposition~\ref{prop:dr}, we are reduced to proving the following statement: Given
an extension $K/k$ and a $G$-torsor $T$ over $K$, the twists ${ }^TX$ and ${ }^TY$
are stably birational over $K$. 

Consider ${ }^T(X\times Y)$, which is isomorphic to ${ }^TX \times { }^TY$, by \cite[Cor.~3.4(a)]{DR}.
We also have ${ }^T(G\cdot \Sigma_X)$ and ${}^T(G\cdot \Sigma_Y)$, which are still
odd-degree multisections over ${ }^TX$ and ${ }^TY$ respectively.  Springer's Theorem
implies that the projections of
${ }^T(X\times Y)$ to ${ }^TX$ and ${ }^TY$ have sections defined over $K$.
Thus ${ }^T(X\times Y)$ is stably birational to both ${ }^TX$ and ${ }^TY$ over 
$K$, and ${ }^TX$ and ${ }^TY$ are stably birational to each other
(see the proof of Prop.~\ref{prop:getstabrat}).
\end{proof}

\section{Equivariant Witt rings}
\label{sect:wring} 

Fix a character $\lambda:G \ra \bG_m$. Consider pairs $(V,q)$, where $V$ is a 
$G$-representation and $q\in \Sym^2(V^{\vee})\otimes L$ is a $G$-quadratic form 
with character $\lambda_q=\lambda$. A {\em homomorphism} of such pairs $(V_1,q_1)$
and $(V_2,q_2)$ is a 
homomorphism of representations $\phi:V_1\rightarrow V_2$ such that
$q_2\circ \phi = q_1$.  
The {\em sum} of $(V_1,q_1)$ and $(V_2,q_2)$ with
$\lambda_{q_1}=\lambda_{q_2}=\lambda$ is defined
$$(V_1,q_1) \oplus (V_2,q_2) = (V_1\oplus V_2, q_1 \oplus q_2);$$
we have $\lambda_{q_1 \oplus q_2}=\lambda$.  

\begin{defi}
Let $G$ be a finite group and $\lambda:G \ra \bG_m$ a character.
The {\em Witt group} $\Witt(G,\lambda)$ is the semigroup of non-degenerate $(V,q)$ with $\lambda_q=\lambda$ with sum operation
as defined above, modulo the hyperbolic forms. 
By convention, $\Witt(G,\lambda)=0$ if there are no such pairs.
\end{defi}
Corollaries~\ref{coro:anihyp} and \ref{coro:struct}
guarantee the group structure is well-defined. 
Concretely, it is the $\bF_2$ vector space generated by isomorphism classes
of irreducible representations $W$ with a symmetric 
$$W \stackrel{\sim}{\lra} W^{\vee} \otimes L$$
where $L$ is the one-dimensional representation associated with $\lambda$.

Suppose that $(V,q)$ and $(W,r)$ are quadratic forms with associated symmetric
homomorphisms
$$s:V \lra V^{\vee} \otimes L, \quad
t:W \lra W^{\vee} \otimes M.$$
Tensoring gives a symmetric
$$s\otimes t: V\otimes W \lra (V\otimes W)^{\vee} \otimes (L \otimes M),$$
which is an isomorphism if $s$ and $t$ are isomorphisms.  
We write 
$$(V,q) \otimes (W,r) = (V\otimes W, q\otimes r)$$
where $q\otimes r$ is the quadratic form associated with $s\otimes t$, with values in $L\otimes M$. 
Tensoring with a hyperbolic form yields a hyperbolic form, thus tensor product yields 
a well-defined
$$\Witt(G,\lambda) \times \Witt(G,\mu) \lra \Witt(G,\lambda \mu),$$
where $\lambda,\mu:G \rightarrow \bG_m$ are the characters associated with $L$ and $M$.
Let ${\bf 1} \in \Witt(G,1)$ denote the non-degenerate quadratic form on the trivial irreducible representation, 
which serves as the unity.

\begin{defi}
The {\em Witt ring} is the $\Hom(G,\bG_m)$-graded algebra
\begin{equation} \Witt(G)=\oplus_{\lambda} \Witt(G,\lambda) 
\label{eqn:grading}
\end{equation}
\end{defi}

We are primarily interested in the projective
geometry of quadrics. Consider a character
$$\nu:G \ra \bG_m$$
with associated one-dimensional representation $N$.
Given a $G$-equivariant $(V,q)$ with character 
$\lambda$, tensoring gives $(V\otimes N,q)$
with character $\lambda \nu^2$.
This is compatible with direct sums and takes
hyperbolic forms to hyperbolic forms. 
Let $\PWitt(G)$ denote the quotient of 
$\Witt(G)$ by this action of $\Hom(G,\bG_m)$. 
Note that 
\begin{align*}
& \PWitt(G) = \oplus_{[\lambda]}
\PWitt(G,[\lambda]) \\
&[\lambda] \in \cok\left(\Hom(G,\bG_m) \stackrel{\cdot 2}{\lra} \Hom(G,\bG_m)\right) 
\end{align*}
i.e., with terms indexed by similitude class.
Each summand is the quotient of $\Witt(G,\lambda)$
under the action of $\Hom(G,\mu_2)$, interpreted
one-dimensional representations with the canonical
invariant quadratic form.

\begin{rema}
The {\em character grading} (\ref{eqn:grading}) is
distinct from filtrations associated with the {\em augmentation ideal} of virtual quadratic forms of rank zero.
\end{rema}

\subsection*{Cyclic groups}
Let $G=\left<g\right>$ be a cyclic group of order $2^n$ with $n\ge 1$; 
fix a primitive
$2^n$-th root of unity $\zeta$.
For $i \in \bZ/2^n\bZ$, let $L_i$ denote the one-dimensional
representation with coordinate $x_i$, where $g$ acts via $\zeta^i$.

The quadric in the regular representation
$$X=\{x_0^2+x_1x_{2^n-1}+\cdots + x_{2^{n-1}}^2=0\} \subset \bP(\oplus_{i=0}^{2^n-1}L_i)$$ is 
linearizable due to the existence of fixed points, e.g.,
$[0,1,0,\ldots,0]$. The discriminant is non-trivial
so $X$ has no $G$-invariant maximal isotropic subspace.
It does have isotropic subspaces of dimension 
$2^{n-1}-1$, e.g.,
$$x_0=x_1=\cdots =x_{2^{n-1}}=0.$$
Thus all $G$-invariant linear sections $Y\subset X$
with $\dim(Y)\ge 2^{n-1}$ are equivariantly
stably birational to $X$, by Proposition~\ref{prop:whyPfister}, hence stably linearizable.

Invariant quadratic forms cover the degree-zero piece
of the grading (\ref{eqn:grading}).
The covariant quadratic forms with 
$$\lambda:G \ra \bG_m,\quad \lambda(g)=\zeta,$$
may be written
$$\sum_{i=0}^{2{n-1}-1} b_i x_i x_{1-i},$$
all of which are hyperbolic. 

The Witt ring $\Witt(G)$ 
only has terms of degree zero. Setting
$${\bf 1}=(L_0,x_0^2), \quad w=(L_{2^{n-1}},x_{2^{n-1}}^2)$$
$$\Witt(G)=\bF_2[w]/\left<w^2-{\bf 1}\right>.$$

Note that the Witt ring does not coincide with the $\bF_2$-cohomology of $G$.  Recall that \cite[Prop.~4.5.1]{CarlsonBook}:
$$\rH^*(G,\bF_2)=\begin{cases} \bF_2[z],  \deg(z)=1 & \text{ if } |G|=2 \\
                               \bF_2[\eta,z]/\left<\eta^2\right>, 
                                    \deg(\eta)=1,\deg(z)=2 
                                    & \text{ if } |G|=2^n>2.
                    \end{cases}
$$

\subsection*{2-elementary groups}
Let $G=\left<g_1,\ldots,g_n\right>$ be the group $C_2^n$.
Write $L_{i_1\ldots i_n}, i_j=0,1$, for the 
one-dimensional representation with coordinate $x_{i_1\ldots i_n}$,
such that $g_j$ acts via $(-1)^{i_j}$. Invariant quadratic
forms take the form
$$\sum_{I=(i_1\ldots i_n)} a_I x_I^2,$$
where the sum is over a subset of characters.
These are anisotropic provided the form is non-degenerate, i.e.,~the coefficients $a_I$ are all non-zero -- see Section~\ref{sect:ASPF}
for discussion.
Fixing a non-trivial 
$$\lambda \in \Hom(G,\mu_2)\simeq \Hom(G,\bG_m),$$
covariant forms of degree $\lambda$ may be written
$$\sum_I b_I x_{\lambda} x_{\lambda-I}.$$
These forms contain no squares, are 
hyperbolic, and admit fixed points, e.g.,~where all
but one of the variables vanish. Setting 
$$w_1=(L_{10\ldots 0},x_{10\ldots 0}^2),\, 
w_2=(L_{010\ldots 0},x_{010\ldots 0}^2),\, \ldots,\, 
w_n=(L_{0\ldots 01},x_{0\ldots 01}^2),
$$
we find
$$\Witt(G)=\bF_2[w_1,\ldots,w_n]/\left<w_1^2-{\bf 1},\ldots,
w_n^2 - {\bf 1}\right>.$$
The augmentation ideal is
$$
\left<w_1 - {\bf 1},\ldots, w_n-{\bf 1}\right>.
$$
Again, the Witt ring is not isomorphic to 
$$
\rH^*(G,\bF_2) \simeq \bF_2[x_1,\ldots,x_n].
$$

The ring $\Witt(G)$ is an Artinian local $\bF_2$-algebra;
it is a complete intersection and thus Gorenstein,
with socle generated by 
$$
\prod_{j=1}^n (w_j - {\bf 1}).
$$
This element arises from the quadratic form on the 
regular representation of $G$
$$q=\sum_{I=(i_1,\ldots,i_n)}(-1)^{|I|}
x_I^2, \quad |I|=i_1+\cdots+i_n.$$

\subsection*{Stable Witt rings}
From the stable birational perspective, demanding a $G$-invariant
isotropic subspace is overly restrictive. Suppose that
$X=\{q=0\}$ is a smooth irreducible quadric hypersurface
with a generically-free action of $G$. Recall
\cite[\S 10]{DR} that the following conditions are
equivalent (see also Proposition~\ref{prop:dr}):
\begin{itemize}
\item{$X$ is stably linearizable;}
\item{$\OGr(\ell,q)$ has a stably linearizable component for some $\ell> 0$;}
\item{for each $G$-torsor $T$ over a field $K$, the twist 
${ }^TX$ has $K$-rational points and is $K$-rational.}
\end{itemize}
To form the {\em stable} Witt ring we should
dispose of equivariant quadratic forms satisfying any of these
conditions. 
\begin{defi}
The {\em stable Witt ring} 
$$\SWitt(G) = \oplus_{\lambda} \SWitt(G,\lambda)$$
is the quotient of $\Witt(G)$ obtained by 
setting the following quadratic forms equal to zero: Non-degenerate $(V,q)$, with $\dim(V)=2m$ and trivial
discriminant, such that $\OGr(m,q)_{\pm}$ is
stably linearizable.  
\end{defi}
\begin{prop}
The stable Witt ring is well-defined. Each element
of $\SWitt_{\lambda}(G)$ may be represented by a 
covariant anisotropic quadratic form over the function 
field of a representation of $G$.
\end{prop}
\begin{proof}
The two components of the maximal isotropic Grassmannian are equivariantly stably birational by Proposition~\ref{prop:maxstabrat}, so the choice of 
component does not affect the definition.

The existence of $G$-invariant maximal isotropic subspaces
implies stable linearizability of the components of
maximal isotropic subspaces by Corollary~\ref{coro:fixedtorat}.
Thus elements trivial in $\Witt(G)$ are evidently trivial
in $\SWitt(G)$. 

We check that the stipulated elements form an ideal.
Suppose that $(V,q)$ and $(W,r)$ represent elements of 
$\Witt(G,\lambda)$, of dimensions $2m$ and $2n$, such that
$\OGr(m,q)_+$ and $\OGr(n,r)_+$ are stably linearizable.
Taking the direct sum $(V\oplus W,q\oplus r)$,
we get a $G$-equivariant 
$$\OGr(m,q) \times \OGr(n,r) \ra \OGr(m+n,q\oplus r)$$
hence the components of the target are also stably 
linearizable by Corollary~\ref{coro:fixedtorat}.
Now assume that $(V',q')$ represents an arbitrary element
of $\Witt(G,\mu)$ of dimension $n'$. We have an
equivariant morphism
$$\OGr(m,q) \rightarrow \OGr(mn',q\otimes q')$$
taking the isotropic subspace $\Lambda$ to $\Lambda \otimes V'$,
which is isotropic by the dimension of $q\otimes q'$.
Again, Corollary~\ref{coro:fixedtorat} gives the desired
conclusion.

For the last statement, take $(V,q)$ and choose $\ell$ maximal such that $\OGr(\ell,q)$ admits an equivariant rational map from a linear representation of $G$. The equivariant version of Witt's Lemma
(Proposition~\ref{prop:getperp}), 
applied over that representation, gives the desired quadratic form: 
the orthogonal complement of the $\ell$-dimensional
isotropic subspace over a stably linearizable base.
Proposition~\ref{prop:Wittdimred} guarantees this is 
anisotropic.  
\end{proof}

\begin{exam}
The simplest situation where these notions differ is for groups $G$
with representations
$$\rho:G \ra \GL_2$$
injecting to $\PGL_2$ but lacking fixed points. 
The symmetric group $G=\fS_3$ is an example. 
Let ${\mathbf 1}$ denote the trivial representation, $L$ the sign representation, and $V_2$ the two-dimensional
irreducible representation; all have non-degenerate $\fS_3$-invariant quadratic forms. Fix corresponding
generators
$$\mathbf{1},w,v \in \Witt(\fS_3)$$
so that
$$\Witt(\fS_3)=\bF_2[w,v]/\left<w^2-\mathbf{1},
wv-v,v^2-v-w-\mathbf{1}\right>.$$
Consider the non-degenerate quadratic form 
$$(V,q), \quad V = V_2\otimes V_2 \simeq V_2 \oplus \mathbf{1} \oplus L,$$
where $q$ is the tensor square of the standard invariant form 
on $V_2$. This is non-trivial in $\Witt(\fS_3)$: The locus
$$X = \{q=0\} \simeq \bP(V_2) \times \bP(V_2)$$
which lacks invariant isotropic subspaces. On the other
hand, $\OGr(2,q)_{\pm} \simeq \bP(V_2)$ which is linearizable;
thus $(V,q)$ is trivial in $\SWitt(\fS_3)$.
\end{exam}
Another example in this vein may be found after
Remark~\ref{rema:anisoRMnotSP}.

\subsection*{Notes on the literature}
Since Milnor's work \cite{Milnor}, cohomological invariants have shaped our conceptual understanding of quadratic form over fields.  Invariants of $G$-equivariant quadratic forms on {\em real} 
representations, taking values in $\rH^*(G,\bF_2)$, have been studied 
as well. We refer the reader to
\cite{GKT} for axiomatic approaches to Stiefel-Whitney classes
and Kahn \cite{Kahnreg1,Kahnreg2,Kahnreg3}
for computations for regular 
representations of $2$-groups and discussion
of filtrations in terms of
these classes.  
Another approach to equivariant Stiefel-Whitney classes, via intersection theory
and the Grothendieck-Riemann-Roch,
may be found in \cite{KozlowskiPAMS,KozlowskiRIMS}.
Guillot \cite{Guillot} has computed invariants
in numerous examples, using the 
comprehensive tables of group cohomology
with $\bF_2$ coefficients in \cite{CarlsonBook}.
Totaro \cite{Totarobook} considers 
cycle-class maps from the perspective of
equivariant Chow groups.

Based on our computations above, there is not a direct dictionary between Witt
rings and group cohomology, as one might
expect from the theory over fields. 

\section{Equivariant stably Pfister forms}
\label{sect:equipf}

\subsection*{Twisting construction}
Let $(V,q)$ be a non-degenerate $G$-quadratic form with 
$\dim(V)=n$; write $X=\{q=0\} \subset \bP(V)$
for the associated projective quadric hypersurface
and $\Cone(X) \subset V$ for the cone over $X$.
Choose an irreducible $G$-variety $U$ such that
$G$ acts freely on an open dense $U_0\subset U$
with field of invariants $F=k(U)^G$.  We have a
$G$-invariant hypersurface
$$\Cone(X) \times U_0 \subset V\times U_0.$$
Taking quotients yields an inclusion of twists
$${ }^{U_0} (\Cone(X)\times U_0) := G \backslash (\Cone(X)\times U_0) \subset G \backslash (V\times U_0)=: { }^{U_0}V.$$
Restricting to the generic point, 
$$
V':=  ({ }^{U_0}V)|_F\simeq \bA^n_F,
$$
by Hilbert Theorem 90 (or the No-Name Lemma), and
$$Z:= { }^{U_0}(\Cone(X)\times U_0)|_F$$
is a non-degenerate affine quadric hypersurface over $F$.
Write $Z=\{q'=0\}$, where $q'$ is a non-degenerate
quadratic form over $F$, well-defined up to a scalar. 
When $q$ is $G$ invariant, we may take $q'=q$, 
re-interpreted as
a quadratic form on $\bA^n_F$.

In summary, twisting takes a $G$-quadratic form to 
a quadratic form over a $G$-invariant field $F$, determined up to scalar; there is a canonical choice provided the quadratic
form is invariant. 

\subsection*{Compatibility with purely transcendental extensions}
From now on, $U$ is a linear representation of $G$
acting freely on a nonempty $U_0 \subset U$.
Recall that {\em general Pfister forms} are non-zero
multiples of Pfister forms \cite[\S{6A}]{EKMbook}.
\begin{prop} \label{prop:puretran}
Let $(V,q)$ be a non-degenerate $G$-quadratic form.
Suppose $U_1$ and $U_2$ are linear representations of $G$
that are generically free, with function fields
of invariants $F_1$ and $F_2$. Write $(V'_1,q'_1)$ and
$(V'_2,q'_2)$ for the corresponding non-degenerate
quadratic forms over $F_1$ and $F_2$.
\begin{itemize}
\item{$(V'_1,q'_1)$ has an isotropic $r$-dimensional linear subspace over $F_1$ iff
$(V'_2,q'_2)$ has an isotropic $r$-dimensional linear subspace over $F_2$.}
\item{$(V'_1,q'_1)$ is general Pfister over $F_1$ iff
$(V'_2,q'_2)$ is general Pfister over $F_2$.}
\end{itemize}
\end{prop}
Our strategy is similar to the approach in \cite[{\S}2]{KrTs26}.
\begin{proof}
The representation $U_1\oplus U_2$ is also generically free;
choose open $U_{12;0} \subset U_{12}$ mapping dominantly to
$U_{1;0} \subset U_1$ and $U_{2;0} \subset U_2$. The group
$G$ acts freely on all these open subsets. 
By the No-Name Lemma,
$$G \backslash U_{1;0} \longleftarrow 
G \backslash U_{12;0} \longrightarrow G\backslash U_{2;0}$$
are birational to affine-space bundles. Thus the invariant field $F_{12}$ is purely transcendental
over $F_1$ and $F_2$.

The first assertion follows because the existence of 
isotropic subspaces does not change under purely transcendental
field extension \cite[1.21,7.15]{EKMbook}. 

For the second assertion, we use the results of \cite[\S{23}]{EKMbook} especially Cor.~23.4: General Pfister forms fall into two disjoint classes:
\begin{itemize}
\item{hyperbolic forms in $2^s$ variables;}
\item{anisotropic forms in $2m>0$ variables admitting 
a rational map 
$$X=\{q=0\} \dashrightarrow \OGr(m,q)$$
from the projective quadric to the variety of maximal isotropic subspaces.}
\end{itemize}
In the former case, the first assertion gives our result.
We therefore focus on the anisotropic case. 
Proposition~\ref{prop:linalg} implies that the 
following conditions are equivalent, for $m>1$ over a field $k$:
\begin{itemize}
\item{
the existence of a rational map
$$X \dashrightarrow \OGr(m,q);$$}
\item{the existence of a rational point of $\OGr(m,q)$
over $k(X)$;}
\item{$X$ is stably birational to $\OGr(m,q)_{\pm}$.}
\end{itemize}
When $m=1$, $X=\OGr(1,q)$ and these conditions are automatic.
Note that the choice of $\OGr(m,q)_{\pm}$ is 
immaterial by Proposition~\ref{prop:vq}.  

Since passing to a purely transcendental extension does not
change the dimensions of isotropic subspaces, all three
conditions are stable under such extensions.
\end{proof}

\subsection*{Stably Pfister equivariant quadratic forms}

\begin{defi} \label{defi:stabPfister}
The $G$-quadratic form $(V,q)$ is {\em stably Pfister} if
$(V',q')$ is a general Pfister form over $F=k(U)^G$
for some generically-free linear representation $U$ of $G$.
\end{defi}
Abusing terminology, we say that $X=\{q=0 \}\subset \bP(V)$
is stably Pfister as well.
Proposition~\ref{prop:puretran} means this is 
independent of the choice of representation $U$.

Recall from Proposition~\ref{prop:dr} -- see \cite[\S 10]{DR} -- that $X$ is stably linearizable if and only if every twist admits a rational point. Thus the following conditions are equivalent
for $m>1$:
\begin{itemize} 
\item[(1)] $X$ and $\OGr(m,q)_{\pm}$ are $G$-equivariantly stably birational;
\item[(2)] for every $G$-torsor $T$
defined over any extension $K/k$, the twists
$${ }^T X,\quad  { }^T\OGr(m,q)_{\pm}$$
are stably birational over $K$;
\item[(3)] there exists a linear representation
$U$ of $G$ for which $G$ acts freely on a dense open subset
$U_0 \subset U$, with field of invariants $F=k(U)^G$,
such that the twists
$$
{ }^{U_0} X,\quad  { }^{U_0}\OGr(m,q)_{\pm}
$$
are stably birational over $F$.
\end{itemize}

One motivation is to emulate the theory of {\em Pfister neighbors}:
Given an anisotropic Pfister form $q$ in $2^n$ variables over $k$ 
and a subform $Q$ in at least $2^{n-1}+1$ variables, then the quadric hypersurfaces
$$\{Q=0\} =:Y \subset X:= \{q=0\}$$
are stably birational over $k$ \cite[23.10]{EKMbook}.  
We obtain many stable birational
equivalences of quadrics with non-trivial birational geometry.
Proposition~\ref{prop:whyPfister} is the corresponding
equivariant statement.
Our notion of ``stably Pfister'' gives stable birational
equivalence among neighbors, provided the
actions are generically free. 

We record some basic properties:
\begin{prop} \label{prop:stabpfisbasic}
Let $(V,q)$ be a non-degenerate $G$-quadratic form.
If $(V,q)$ is stably Pfister for $G$ then it is stably
Pfister for subgroups of $G$.  

Suppose that $(V,q)$ and $(V',q')$ are non-degenerate
quadratic forms for groups $G$ and $G'$.
Assume that 
$$
X=\{q=0\} \subset \bP(V), \quad
X'=\{q'=0\} \subset \bP(V')
$$ 
are stably Pfister for $G$ and $G'$, respectively. 
Then 
$$
X \otimes X' = \{q\otimes q'=0\} \subset \bP(V\otimes V')
$$
is stably Pfister for $G\times G'$.
\end{prop}
Recall the Pfister condition is compatible with field
extensions and tensor products. 
Indeed, this follows from equivalences supporting
Definition~\ref{defi:stabPfister}: An equivariant form
is stably Pfister iff all its twists are Pfister. 
However, the definition of Pfister forms over fields
\cite[\S 4.B]{EKMbook} makes clear that tensor products of Pfister forms are also Pfister. 

\subsection*{Examples} 
\label{sect:ASPF}

\subsubsection*{2-elementary groups}
Let $G=C_2^n$ and consider the natural diagonal form
on the regular representation
$$q=\sum_{I=(i_1\ldots i_n)}(-1)^{|I|} x_I^2.$$
For example, when $n=3$ we have
$$q = x_{000}^2 - x_{100}^2 - x_{010}^2 - x_{001}^2
+ x_{110}^2 + x_{101}^2 + x_{011}^2 - x_{111}^2.$$
Let $G$ act on $\bA^n_{t_1,\ldots,t_n}$, where the $i$th factor
of $C_2^n$ acts via $\pm 1$ on $t_i$ and trivially on the
other coordinates; write 
$$T:=\{t_1\cdots t_n \neq 0\}\subset \bA^n$$
for the locus where $G$ acts freely. Set $a_i=t_i^2$ so we may write
$${ }^Tq = \sum_I a^I x_I^2,$$
e.g., for $n=3$
$$
x_{000}^2 - a_1 x_{100}^2 - a_2x_{010}^2 - a_3x_{001}^2
+ a_1a_2x_{110}^2 + a_1a_3x_{101}^2 + a_2a_3x_{011}^2 - 
a_1a_2a_3 x_{111}^2.
$$
This is a generic Pfister form over the field $F=k(a_1,\ldots,a_n)$.
Thus $X=\{q=0\}$ is stably Pfister as a $G$-equivariant
quadric hypersurface.  

These forms are not stably linearizable -- anisotropic over the function
field:  
Equivariantly, this follows because $X=\{q=0\}$ has no fixed points
under $G$, see, e.g., \cite[\S 5 and Appendix]{RY} relating fixed points and stable linearizability for abelian groups.
Over function fields this is classical; one reference
is the theory of normic forms \cite[p.~377]{LangQAC}. 

Proposition~\ref{prop:whyPfister} immediately gives
\begin{prop}
Consider a quadric
$$Y_S=\{\sum_{I \in S} (-1)^{|I|} x_I^2=0\} \subset \bP^{|S|-1}$$
where $S \subset \{0,1\}^n$ satisfies $|S|>2^{n-1}.$
Then $Y_S$ is $G$-stably birational to $X$.  
\end{prop}
Again the $G$-actions on $X$ and $Y_S$ 
lack fixed points so they are not stably linearizable.

\subsubsection*{Pfister surfaces}
Let $(V,q)$ be a $G$-invariant quadratic form of dimension $\dim(V)=4$;
assume that the discriminant vanishes and the resulting
action on $X=\{q=0\}$ is generically free.

Then $X\simeq R_+\times R_-$ 
where $R_+$ and $R_-$ are conics, associated with 
projective representations
$$\rho_{\pm}:G \rightarrow \PGL_2$$
whose Amitsur invariants $\alpha_+$ and $\alpha_-$ 
sum to zero. Since these are two-torsion, they are 
equal; the resulting class $\alpha \in \rH^2(G,\bG_m)[2]$
is the {\em Clifford invariant} of $X$ under $G$.
The projections 
$$\pi_{\pm}:X \ra R_{\pm}$$
guarantee that $X$ is stably Pfister --
provided that the surface is not stably linearizable.

\begin{prop} 
The $G$-surface
$X$ is stably linearizable iff $\alpha=0$.
\end{prop}
\begin{proof}

If $\alpha=0$ then $X$ is a product of linearizable $\bP^1$'s,
hence is stably linearizable. 
When $\alpha \neq 0$ then $X$ cannot be 
stably linearizable;
it dominates a variety with non-trivial Amitsur invariant. (See \cite[Prop.~A7]{BCDP} and \cite[\S7]{HTNagoya} for context and detail.)
\end{proof}

This yields examples of $G$-surfaces that are stably Pfister
where $G$ is non-abelian; however, they are tightly
tied to $2$-elementary examples:

\begin{exam}
The representation of the dihedral group of order eight
$$(s,t) \mapsto (t,s), \quad (s,t) \mapsto (is,i^3 t)$$
gives a projective representation 
$$\rho:C_2\times C_2 \rightarrow \PGL_2.$$
Its tensor square $\rho^{\otimes 2}$:
\begin{align*}(s_+s_-,s_+t_-,t_+s_-,t_+t_-) &\mapsto 
(t_+t_-,t_+s_-,s_+t_-,s_+s_-) \\
(s_+s_-,s_+t_-,t_+s_-,t_+t_-) & \mapsto
(-s_+s_-,s_+t_-,t_+s_-,-t_+t_-)
\end{align*}
is a linear representation of $C_2\times C_2$. 
Setting 
\begin{align*}
x_{00}& =s_+t_- + t_+s_-, \quad 
x_{10} =s_+t_- - t_+s_-,\\
x_{01}& =s_+s_- + t_+t_-,\quad 
x_{11} =s_+s_- - t_+t_-,
\end{align*}
we see this is the standard regular representation of $C_2\times C_2$.
\end{exam}

\section{Analysis of unstable Pfister forms}
\label{sect:unstable}

Can we make sense of equivariant Pfister forms without
stabilizing? We propose a notion (Definition~\ref{defn:equmult}) with some reasonable
properties. However, formulating the notion of ``anisotropic''
is subtle; see Remark~\ref{rema:anisoRMnotSP}.

Throughout, we take $G$ be a finite group, $(V,q)$ a non-degenerate $G$-quadratic form, and
$
X=\{q=0\}\subset \bP(V)
$ 
the associated smooth quadric.

\subsection*{Multiplicative forms}
We propose a definition inspired by
Pfister's original formulation \cite[ch.~2]{Pfisterbook}:
Let $(V,q)$ be a non-degenerate quadratic form with similitude group
$\GO(V,q)$ and character 
$$
\lambda: \GO(V,q) \ra \bG_m.
$$
It is strictly multiplicative if there exists a rational map
$$\tau: V \dashrightarrow \GO(V,q)$$
such that 
$$q(\tau(x)\cdot y)=\lambda(\tau(x)) q(y) = q(x)q(y), \quad x,y \in V.$$
We refer the reader to \cite[\S{23}]{EKMbook} for how various definitions of Pfister forms are related.

For invariant forms, there is a natural equivariant analog:
\begin{defi}
\label{defn:equmult}
Let $G$ be a finite group and $(V,q)$ a non-degenerate $G$-invariant
quadratic form.
It is {\em equivariantly multiplicative}
if there exists a $G$-equivariant rational map
$$
\tau:V \dashrightarrow \GO(V,q)
$$
such that 
\begin{equation} \label{taufunc}
q(\tau(x)\cdot y) = q(x)q(y), \quad x,y \in V,
\end{equation}
i.e. $\lambda(\tau(x))=q(x)$.  
\end{defi}
The natural action on similitudes
$$(g,\tau) \mapsto g \tau g^{-1}$$
is compatible with this formula:
$$
q(\tau(g\cdot x)\cdot y)\!=\!q((g \tau(x) g^{-1})\cdot y)\!=\!q(\tau(x) \cdot g^{-1}y)\!=\!
q(x)q(g^{-1}y)=q(x)q(y).
$$

The ``Twisting Construction'' of Section~\ref{sect:equipf}
gives the following:
Suppose that $(V,q)$ is $G$-invariant and equivariantly
multiplicative. 
Let $U_0$ be an affine variety with free $G$-action 
and invariant field $F=k(U_0)^G$. The resulting quadratic form
over $(V',q'=q)$ over $F$ is strictly multiplicative because
$\tau$ descends to $F$.  Thus we obtain:

\begin{prop} \label{prop:EMtoSP}
Equivariantly multiplicative forms are stably Pfister.
Thus the underlying vector space $V$ satisfies 
$\dim(V)=2^e$ for some integer $e$ and the
discriminant is trivial when $e>1$.  
\end{prop}
\begin{ques}
To what extent are stably Pfister forms equivariantly
multiplicative?
\end{ques}

\begin{prop} \label{prop:EMtoRatMap}
Suppose $(V,q)$ is equivariantly multiplicative of rank $2m$
and the $G$-action on $X$ is generically free.  
Then there is an equivariant rational map $X \dashrightarrow \OGr(m,q)$.
\end{prop} 
We refer the reader to Remark~\ref{rema:anisoRMnotSP}
for discussion of the converse.
\begin{proof} 
This is tautologically true for $m=1$.
When $m>1$, $X$ is irreducible and any map factors through $\OGr(m,q)_{\pm}$. We have
already discussed the $m=2$ case:
$X$ is a quadric surface and $\OGr(2,q)_{\pm}$
parametrize its rulings. Thus we focus on $m>2$.  

We use the notations of Section~\ref{sect:compactify}.
The similitude form $\lambda$ yields an effective divisor $$\Lambda=\{\lambda=0\} \subset \bP(\End(V))$$
with proper transform $\Lambda' \subset \oX(V,q)$.
Proposition~\ref{prop:Dm-1Dm}
implies that $\Lambda'$ has two irreducible components
$D_{m-1}$ and $D_m$, each of multiplicity one. 
Proposition~\ref{prop:excmulttwo} gives that the difference
$$
\Lambda - \Lambda' = 2E,
$$
where $E$ is the reduced exceptional divisor of
$$\oX(V,q) \stackrel{\beta}{\longrightarrow} \oPSO(V,q) \subset \bP(\End(V)).$$

The rational map $\tau$ induces 
$$\tau': V \dashrightarrow \oX(V,q),$$
which is defined at the generic point of the
hypersurface $\{q=0\}$ due to the properness of the target.
We are abusing notation, using $\oX(V,q)$ for the
component of the wonderful compactification of $\PGO(V,q)$
containing the image of $\tau$; see Remark~\ref{rema:EOG}
for further context.
The hypersurface $\Cone(X)=\{q=0\}$ is mapped into $\Lambda'$ by
the functional relation (\ref{taufunc}),
which guarantees that $\lambda$ has 
multiplicity one along our hypersurface.  
The No-Name Lemma 
gives a rational section to the cone 
$$\Cone(X) \setminus 0 \ra X$$
and an induced 
$$\tau'': X \dashrightarrow \oX(V,q).$$

We claim that $X$ is mapped to one of the two irreducible components of $\Lambda'$ dominating divisors in the closure 
of $\PSO(V,q) \subset \bP(\End(V))$. 
Proposition~\ref{prop:Dm-1Dm} says that these 
are birational to projective bundles over
$$\OGr(m,q)_+ \times \OGr(m,q)_- ,\text{ for $m$ odd},
$$
or 
$$\OGr(m,q)_{\pm} \times \OGr(m,q)_{\pm},  \text{ for $m$ even.}
$$
Composing $\tau''$
with the projection to one of the factors yields
$$X \dashrightarrow \OGr(m,q)_+ \text{ or } \OGr(m,q)_-, 
$$
as desired.
To see this, note that all the other irreducible components of $\Lambda'$ have {\em even} multiplicity. 
Thus if $\tau'$ took $X$ to these components -- missing the
two distinguished components -- then the
pull-back of $\lambda$ to $V$ would necessarily vanish
to order $\ge 2$ along $X$. This would contradict the
fact that $\lambda$ pulls back to $q$, the defining equation
of $X$.  
\end{proof}

\subsection*{Competing notions of isotropy}

Let $(V,q)$ be a non-degenerate quadratic form in $2m$-variables over a field $k$,
with associated hypersurface $X\subset \bP(V)$. 
Assuming it is anisotropic, it is proportional to a Pfister form if and only if there is a rational map
\begin{equation}
 \label{eqn:ogrsect}
X=\{q=0\} \dashrightarrow \OGr(m,q).
\end{equation} 
Does this make sense equivariantly? We interpret ``anisotropic'' as lacking a
$G$-invariant isotropic linear subspace.  We want that the resulting class has properties similar to Pfister quadrics over fields.

\begin{rema} \label{rema:anisoRMnotSP}
There exists a non-degenerate $G$-quadratic form $(V,q)$
of dimension $2m$,
anisotropic in the sense above, that admits an equivariant rational map (\ref{eqn:ogrsect})
but is not stably Pfister.
\end{rema}

Consider the Frobenius group
$$G = C_5 \rtimes C_4, \quad C_5 = \left< \sigma \right>, C_4 = \left< \tau \right>,$$
with the cyclic group $C_4$ acting via automorphisms on $C_5$
$$\sigma^5=\tau^4=1, \quad \tau \sigma \tau^{-1} = \sigma^2.$$
Let $W=\left<w_1,w_2,w_3,w_4\right>$ denote the irreducible four-dimensional representation
\begin{align*}
 \sigma \cdot w_i &= \zeta^{2^{i-1}} w_i, \quad \zeta=e^{2\pi i/5} \\
\tau \cdot w_1 &= w_4, \quad \tau \cdot w_i = w_{i-1}, \ i=2,3,4.
\end{align*}
The projective space $\bP(W)$ has no fixed points under $G$.

Consider the $G$-quadratic form $(V,q)$:
$$V=\bigwedge^2 W, \quad
q:\bigwedge^2 W \ra \bigwedge^4 W$$
i.e., squaring of elements. The form $q$ is $\sigma$-invariant and transforms by the
sign character on $\left<\tau\right>$.
Write
$$X := \Gr(2,W) = \{q=0\} \subset \bP(\bigwedge^2 W),$$
a quadric fourfold with generically-free $G$-action.
We have a decomposition of $G$-representations
$$\bigwedge^2 W \simeq W \oplus \left< w_1\wedge w_3+ i w_2\wedge w_4\right> \oplus
				\left<w_1\wedge w_3 - i w_2\wedge w_4\right>,$$
where the final two summands are non-isomorphic and not isotropic for $q$. 
Projection from this two-dimensional representation gives a rational map
$$\psi: X \dashrightarrow \bP(W),$$
a conic fibration over the generic point of $\bP(W)$. 
Note that $X$ is anisotropic in that it has no $G$-invariant isotropic subspaces.  

In this situation
$$\OGr(3,q) = \bP(W) \sqcup \bP(W^{\vee})$$
with isotropic spaces given by lines through $w\neq 0 \in W$ and
lines contained in $w^{\vee}=0$.  Thus $\psi$ induces
$$X \dashrightarrow \OGr(3,q)$$
i.e., $X$ is hyperbolic over its function field.  
However, since $q$ has six variables we cannot 
reasonably interpret it to be equivariantly Pfister.  

We close with two remarks:
\begin{itemize}
\item{$X=\{q=0\}=\Gr(2,W)$ is stably linearizable; its universal
rank-two bundle is a $\bP^2$-bundle over $W\setminus 0$.}
\item{The components of $\OGr(3,q)$ are linearizable.}
\end{itemize}
Thus $q$ is ``stably'' isotropic and hyperbolic.

\subsection*{Examples of equivariantly multiplicative actions}
\subsubsection*{2-elementary groups}
\begin{exam}
Let $G=C_2$ with regular representation $W$, where the coordinate $x_0$ is 
invariant and $x_1$ anti-invariant. Then  $q=x_0^2 - x_1^2$
is strictly multiplicative
$$(x_0^2-x_1^2)(y_0^2-y_1^2) = (x_0y_0+x_1y_1)^2 - (x_1y_0+x_0y_1)^2.$$
The mapping
$$\tau(x_0,x_1)\cdot (y_0,y_1) = \left( \begin{matrix} x_0 & x_1 \\ x_1 & x_0 \end{matrix} \right) 
\left( \begin{matrix} y_0 \\ y_1 \end{matrix}\right)$$
is $C_2$-equvariant.
\end{exam}

\begin{exam}
Consider the $G=C_2^2$-equivariant form
$$q=x_{00}^2 - x_{10}^2 - x_{01}^2 +x_{11}^2,
$$
where the subscripts are characters $G \ra \mu_2$.
Setting
\begin{align*}
z_{00} &= +x_{00}y_{00}-x_{10}y_{10}+x_{01}y_{01}+x_{11}y_{11}\\
z_{10} &= -x_{10}y_{00}+x_{00}y_{10} +x_{11}y_{01}+x_{01}y_{11}\\
z_{01} &= +x_{01}y_{00} -x_{11}y_{10} + x_{00}y_{01} + x_{10}y_{11} \\
z_{11} &= -x_{11}y_{00} + x_{01}y_{10} + x_{10}y_{01} + x_{00} y_{11}
\end{align*}
we get the desired multiplicative relation
$$z^2_{00}-z_{10}^2 - z_{01}^2 + z_{11}^2 \mapsto
(x^2_{00}-x_{10}^2 - x_{01}^2 + x_{11}^2)(y^2_{00}-y_{10}^2 - y_{01}^2 + y_{11}^2).
$$
The matrix expressing the $z$-variables in terms of the $y$-variables
has determinant $q^2$.
\end{exam}

\begin{rema}
The choice of signs for the Pfister form is not natural:
The action of $\GL_2(\bZ/2\bZ)$ on $C_2^2$ alters the form e.g. to
$$x_{00}^2 - x_{10}^2 + x_{01}^2 -x_{11}^2 \text{ or }
x_{00}^2 + x_{10}^2 - x_{01}^2 -x_{11}^2.$$
The form 
$$\hat{q}=x_{00}^2 + x_{10}^2 + x_{01}^2 + x_{11}^2$$
is invariant.  
Writing 
\begin{align*}
z_{00} &= +x_{00}y_{00}+x_{10}y_{10}-x_{01}y_{01}-x_{11}y_{11}\\
z_{10} &= -x_{10}y_{00}+x_{00}y_{10} -x_{11}y_{01}+x_{01}y_{11}\\
z_{01} &= +x_{01}y_{00} +x_{11}y_{10} + x_{00}y_{01} + x_{10}y_{11} \\
z_{11} &= +x_{11}y_{00} - x_{01}y_{10} - x_{10}y_{01} + x_{00} y_{11}
\end{align*}
induces
$$(z_{00}^2+z_{10}^2+z_{01}^2+z_{11}^2) \mapsto
(x_{00}^2+x_{10}^2+x_{01}^2+x_{11}^2) 
(y_{00}^2+y_{10}^2+y_{01}^2+y_{11}^2).$$ 
The transformation 
$$\left( \begin{matrix}
x_{00} &x_{10} &-x_{01}&-x_{11}\\
-x_{10}&x_{00}&-x_{11}&x_{01}\\
x_{01}&x_{11}&x_{00}&x_{10}\\
x_{11}&-x_{01}&-x_{10}&x_{00}
\end{matrix} \right)$$
has determinant $\hat{q}^2$.

However, the construction here is also non-canonical. For example,
reversing the signs of both $x_{01}$ and $y_{01}$ gives the
matrix:
\begin{equation} \label{eqn:mat1}
\left( \begin{matrix}
x_{00} &x_{10} &-x_{01}&-x_{11}\\
-x_{10}&x_{00}&x_{11}&-x_{01}\\
-x_{01}&x_{11}&-x_{00}&x_{10}\\
x_{11}&x_{01}&x_{10}&x_{00}
\end{matrix} \right)
\end{equation}
and then reversing the signs of $x_{10}$ and $y_{10}$ gives:
$$\left( \begin{matrix}
x_{00} &x_{10} &-x_{01}&-x_{11}\\
x_{10}&-x_{00}&x_{11}&-x_{01}\\
-x_{01}&-x_{11}&-x_{00}&-x_{10}\\
x_{11}&-x_{01}&-x_{10}&x_{00}
\end{matrix} \right).$$
We can also reverse the sign of any row of the matrix,
e.g., switching the third row of (\ref{eqn:mat1}) gives
$$
\left( \begin{matrix}
x_{00} &x_{10} &-x_{01}&-x_{11}\\
-x_{10}&x_{00}&x_{11}&-x_{01}\\
x_{01}&-x_{11}&x_{00}&-x_{10}\\
x_{11}&x_{01}&x_{10}&x_{00}
\end{matrix} \right).
$$
\end{rema}

\subsubsection*{Rank two forms}
We consider the functional equation for rank-two quadratic forms $(V,q)$. Here $G$ is cyclic or dihedral
$$C_r=\left<\sigma:\sigma^r=1\right>\quad
D_r=\left<\sigma,\varrho:\sigma^r=\varrho^2=1, \varrho\sigma\varrho^{\-1}=\sigma^{-1}\right>.$$ 
In suitable coordinates, $q=x_0x_1$ with
$$\sigma:(x_0,x_1) \mapsto (\zeta x_0, \zeta^{-1}x_1), \quad
\zeta \text{ primitive $r$th root of unity}
$$
and
$$\varrho: (x_0,x_1) \mapsto (x_1,x_0).$$

The similitude group 
$$\GO(V,q) \simeq \bG_m^2 \rtimes C_2$$
where the first factor acts diagonally and the second interchanges $x_0$ and $x_1$.
The diagonal sub-torus of the orthogonal group acts trivially on the identity component of $\GO(V,q)$. When $\tau$ maps $V$ to 
the identity component, we have
\begin{equation} \label{eqn:Crtau}
\tau(x_0,x_1) = \left(\begin{matrix} \tau_{11}(x_0,x_1) & 0 \\
                                       0 & \tau_{22}(x_0,x_1)
                                       \end{matrix} \right),
\end{equation}
where the $\tau_{ii}(x_0,x_1)$ are invariant under
$C_r=G\cap \bG_m$. 

Now for $G=C_r$ we may take 
$$\tau_{11}(x_0,x_1)=x_0x_1, \quad \tau_{22}(x_0,x_1)=1$$
which is equivariantly multiplicative.
Note, however, that when $r$ is even then $\tau$
is an even function of $(x_0,x_1)$, unchanged
under $\pm I$.  
When $G$ contains $\varrho$, we have
$\tau_{11}(x_1,x_0)=\tau_{22}(x_0,x_1)$.
In the odd case, with $r=2m+1$, we may take 
$$\tau(x_0,x_1) = \left( \begin{matrix} x_0^{m+1}/x_1^m & 0 \\
                                        0 & x_1^{m+1}/x_0^m
                                \end{matrix} \right)$$
which gives multiplicativity for the dihedral group $D_{2m+1}$. Here $\tau$ is odd as a function
of $(x_0,x_1)$.

\begin{prop}
Let $G=\left<\sigma,\varrho\right> \simeq C_2\times C_2$, with 
$\sigma=-I$. Then $(V,q)$ is not equivariantly multiplicative.
The same holds for all dihedral groups $D_{2r}$ with $r>0$.
\end{prop}
\begin{proof}
First, $\varrho$ exchanges the two components of 
$\GO(V,q)$ and commutes with $\sigma$. Thus we may 
assume that $\tau$ takes $V$ to the identity component
as in (\ref{eqn:Crtau}).
As above $\tau_{22}(x_0,x_1)=\tau_{11}(x_1,x_0)$, hence
$$\tau_{11}(x_0,x_1)\tau_{11}(x_1,x_0)=x_0x_1.$$
Consider the quadratic field extension
$$L:=k(x_0^2,x_0x_1) / k(x_0^2+x_1^2,x_0x_1)=:K$$
associated with invariants under $C_2$ and $C_2\times C_2$.
We may write $L=K(\gamma)$ with $\gamma=x_0^2$ satisfying
$$\gamma^2 - \gamma(x_0^2+x_1^2) + (x_0x_1)^2.$$
Writing $\tau_{11}=(u+\gamma v)/w$ with $u,v,w$ polynomials in 
$x_0^2+x_1^2$
and $x_0x_1$, we get the homogeneous equation
$$u^2-(x_0^2+x_1^2)uv + (x_0x_1)^2v^2 = x_0 x_1 w^2,$$
a smooth conic over $K$. Now $K$ is purely transcendental
with generators $A=x_0x_1$ and $B=x_0^2+x_1^2$ and our conic
diagonalizes to
$$\hat{u}^2 + (4B^2-A^2) \hat{v}^2 = Aw^2.$$
Localizing along $\left<2B-A\right>$,
where $\varpi$ is the uniformizer of the resulting DVR,
yields
$$\hat{u}^2 + \varpi (\varpi+2A) \hat{v}^2 = Aw^2.$$
The resulting residue is $\sqrt{-A}$, which is non-trivial.

This precludes equivariant multiplicativity for larger groups
$G \supset \left<\sigma,\varrho\right>$.
\end{proof}

\subsection*{Inductive results}
We retain the standing assumptions of
Section~\ref{sect:unstable}. 
Our next result is inspired by the original argument
that Pfister forms are strictly multiplicative
\cite[ch.~2]{Pfisterbook}:

\begin{prop}
Assume $(V,q)$ is non-degenerate, $G$-invariant, and 
equivariantly multiplicative with rational map
$$\tau:V \dashrightarrow \GO(V,q)$$
that is odd as a function of $V$.
Let $W=\operatorname{span}(1,\alpha)$ be the regular representation of $C_2$, with $\alpha$ a non-trivial character of $C_2$. Write
$V\otimes W$ for the resulting representation of $G\times C_2$ and 
$$\wQ(x\otimes 1+x'\otimes \alpha) = q(x) - q(x')$$
for the induced quadratic form.

Then this is equivariant multiplicative for $G\times C_2$, i.e.~there is a linear transformation
$$\wT(x\otimes 1 + x'\otimes \alpha): V\otimes W \rightarrow V\otimes W,$$
with coefficients rational functions in $x$ and $x'$,
with
$$\wQ(\wT(x\otimes 1 + x'\otimes \alpha)\cdot
(y\otimes 1 + y'\otimes \alpha))=
\wQ(x\otimes 1 + x'\otimes \alpha)
\wQ(y\otimes 1 + y'\otimes \alpha).$$
Moreover, $\wT$ is odd as a function on $V\otimes W$.
\end{prop}
\begin{proof}
Set 
\begin{equation}
\wT(x \otimes 1 + x' \otimes \alpha)
= \left(
\begin{matrix} \tau(x) & -\tau(x') \\
        -\tau(x') & \frac{q(x)}{q(x')}\tau(x')\tau(x)^{-1}\tau(x')
        \end{matrix} \right)
\label{eqn:wT}
\end{equation}
a matrix expressed in terms of $\{1,\alpha \}.$
We expand out 
$$\wQ\left(\wT(x\otimes 1 + x'\otimes \alpha)\cdot
(y\otimes 1 + y'\otimes \alpha)\right),$$
using the functional relation along with the fact
that $q$ is quadratic, as follows:
\begin{align*}
q(\tau(x)y - \tau(x')y') - 
q(-\tau(x')y+\frac{q(x)}{q(x')}\tau(x')\tau(x)^{-1}\tau(x')y')  \\
= 
q(\tau(x)y - \tau(x')y') - 
q(x') q(-y+\frac{q(x)}{q(x')}\tau(x)^{-1}\tau(x')y')\\
=
q(\tau(x)y - \tau(x')y') - 
\frac{q(x')}{q(x)} q(-\tau(x)y+\frac{q(x)}{q(x')}\tau(x')y') \\
= q(\tau(x)y - \tau(x')y') - \frac{1}{q(x)q(x')}
q(-q(x')\tau(x)y+q(x)\tau(x')y').
\end{align*}
Writing $b(,)$ for the symmetric bilinear form 
associated with $q$, this becomes
\begin{align*}
&q(\tau(x)y) - 2b(\tau(x)y,\tau(x')y')+q(\tau(x')y') \\
&-\frac{q(q(x')\tau(x)y)-2b(q(x')\tau(x)y,q(x)\tau(x')y') + q(q(x)\tau(x')y')}{q(x)q(x')}.
\end{align*}
The bilinear terms cancel, so clearing denominators
yields the simplification
$$q(\tau(x)y) + q(\tau(x')y') - q(\tau(x)y)\frac{q(x')}{q(x)}
- q(\tau(x')y') \frac{q(x)}{q(x')}.$$
Applying the functional relation again gives
$$q(x)q(y) + q(x')q(y') - q(y)q(x') - q(y')q(x)
=(q(x)-q(x'))(q(y)-q(y')),$$
as desired.

Equivariance of $\wT$ under the action of $G$ follows from the equivariance of $\tau$. Equivariance under
$C_2$ follows because $\tau$ is odd:
$$\left(
\begin{matrix} \tau(x) & -\tau(-x') \\
        -\tau(-x') & \frac{q(x)}{q(x')}\tau(-x')\tau(x)^{-1}\tau(-x')
        \end{matrix} \right)
        =
        \left(
\begin{matrix} \tau(x) & \tau(x') \\
        \tau(x') & \frac{q(x)}{q(x')}\tau(x')\tau(x)^{-1}\tau(x')
        \end{matrix} \right); 
$$
this is the conjugation of our original matrix by
$$\alpha = \left( \begin{matrix} 1 & 0 \\ 0 & -1 \end{matrix}
\right).$$
Formula (\ref{eqn:wT}) also shows that $\wT$ is odd 
as a function of $(x,x')$.
\end{proof}

Proposition~~\ref{prop:stabpfisbasic} shows that stably Pfister equivariant quadratic forms behave well under tensor 
product.

\begin{ques}
Suppose that $(V,q)$ and $(W,r)$ are equivariantly multiplicative for $G$ and $H$ respectively.  
Is $(V\otimes W, q\otimes r)$ equivariantly multiplicative for $G\times H$?
\end{ques}

\bibliographystyle{alpha}
\bibliography{PE}

\end{document}